\documentclass[12pt]{amsart}

\usepackage{graphicx,indentfirst}
\usepackage{amsmath,amssymb,mathrsfs}
\usepackage{amsthm,amscd}
\usepackage{verbatim}
\usepackage{appendix}
\usepackage{enumitem,titletoc}
\usepackage{appendix}
\usepackage{imakeidx}
\makeindex[columns=2, title=Alphabetical Index]
\usepackage{fancyhdr}
\usepackage{amsfonts,color}
\usepackage[all]{xy}
\usepackage{tikz-cd}
\usepackage{syntonly}
\usepackage{fancyhdr}
\newtheorem{theorem}{Theorem}[section]
\newtheorem{lemma}{Lemma}[section]
\newtheorem{remark}{Remark}[section]
\usepackage[left=2.4cm,right=2.4cm,bottom=2.8cm,top=2.8cm]{geometry}

\newcommand{\abs}[1]{|#1|^2}

\usepackage{tikz}
\usepackage{extarrows}
\usepackage{hyperref}

\def\Xint#1{\mathchoice
{\XXint\displaystyle\textstyle{#1}}%
{\XXint\textstyle\scriptstyle{#1}}%
{\XXint\scriptstyle\scriptscriptstyle{#1}}%
{\XXint\scriptscriptstyle\scriptscriptstyle{#1}}%
\!\int}
\def\XXint#1#2#3{{\setbox0=\hbox{$#1{#2#3}{\int}$ }
\vcenter{\hbox{$#2#3$ }}\kern-.6\wd0}}

\def\dashint{\Xint-}
\newtheorem{prop}{Proposition}[section]
\newtheorem{defn}{Definition}[section]
\newtheorem{corr}{Corollary}[section]

\allowdisplaybreaks

\newcommand{\ddbar}{i\partial\bar\partial}
\newcommand{\ric}{{Ric}}
\newcommand{\innpro}[1]{\langle#1\rangle}
\newcommand{\bk}[1]{\Big(#1\Big)}

\newcommand{\xk}[1]{\big(#1\big)}

\newcommand{\eE}{\mathcal{F}}

\makeindex
\newcommand{\diam}{Diam}

\DeclareMathOperator{\vol}{Vol}

\DeclareMathOperator{\tr}{tr}

\numberwithin{equation}{section}

\hypersetup{
    colorlinks=true,
    citecolor=green,
    filecolor=green,
    linkcolor=blue,
    urlcolor=black
}

\begin{document}

\address{Department of Mathematics, Rutgers University, Piscataway, NJ 08854}

\email{xf35@math.rutgers.edu}
\address{Department of Mathematics, Columbia University, New York, NY 10027}

\email{bguo@math.columbia.edu}

\address{Department of Mathematics, Rutgers University, Piscataway, NJ 08854}
\email{jiansong@math.rutgers.edu}

\title{Geometric estimates for complex Monge-Amp\`ere equations}
\author{Xin Fu\qquad Bin Guo \qquad Jian Song}

\thanks{Research supported in
part by National Science Foundation grants DMS-1406124}

\begin{abstract} We prove uniform gradient and diameter estimates for a family of geometric complex Monge-Amp\`ere equations. Such estimates can be applied to study geometric regularity of singular solutions of complex Monge-Amp\`ere equations. We also prove a uniform diameter estimate for  collapsing families of twisted K\"ahler-Einstein metrics on K\"ahler manifolds of nonnegative Kodaira dimensions.

\end{abstract}

\maketitle

\section{Introduction}  

Complex Monge-Amp\`ere equations are a fundamental tool to study K\"ahler geometry and, in particular, canonical K\"ahler metrics of Einstein type on smooth and singular K\"ahler varieties. Yau's solution to the Calabi conjecture establishes the existence of Ricci flat K\"ahler metrics  on K\"ahler manifolds of vanishing first Chern class by a priori estimates for complex Monge-Amp\`ere equations \cite{Y1}. 

Let $(X, \theta)$ be a K\"ahler manifold of complex dimension $n$ equipped with a K\"ahler metric $\theta$. 
We consider the following complex Monge-Amp\`ere equation 
\begin{equation} \label{cma} 
(\theta+ \ddbar \varphi)^n = e^{ - f} \theta^n,
\end{equation}
where  $f\in C^\infty(X)$ satisfies the normalization condition $$\int_X e^{-f}\theta^n = \int_X \theta^n = [\theta]^n. $$
In the deep work of Kolodziej \cite{K}, Yau's $C^0$-estimate for solutions of equation (\ref{cma}) is tremendously improved by applying the pluripotential theory and it has important applications for singular and degenerate geometric complex Monge-Amp\`ere equations.  More precisely, suppose  the right hand side of equation (\ref{cma}) satisfies the following $L^p$ bound
$$\int_{X} e^{-pf} \theta^n \leq K, ~ \textnormal{for ~some}~ p>1, $$
then there exists  $C=C(X, \theta, p, K)>0$ such that any solution $\varphi$ of equation (\ref{cma}) satisfies the the following $L^\infty$-estimate
$$ \| \varphi - \sup_X \varphi \|_{L^\infty(X)} \leq C. $$
In particular, the equation (\ref{cma}) admits a unique continuous solution in $\textnormal{PSH}(X, \theta)$ as long as $e^{-f} \in L^p(X, \theta^n)$ without any additional regularity assumption for $f$. In \cite{K2, DDGHKZ}, it is shown that the bounded solution is also H\"older continuous and the H\"older exponent only depends only on $n$ and $p$. However, in general the solution is not uniformly Lipschitz continous (see e.g. \cite{DDGHKZ}). 

Complex Monge-Amp\`ere equations are closely related to geometric equations of Einstein type, and in many geometric settings, one makes assumption on a uniform lower bound  of the Ricci curvature. Therefore it is natural to consider the family of volume measures, whose curvature is uniformly bounded below. More precisely, 
we let $\Omega = e^{-f}\theta^n$ be a smooth volume form on $X$ such that
\begin{equation}\label{assump0}
Ric(\Omega) = - \ddbar \log \Omega \geq - A \theta
\end{equation}
for some fixed constant $A\geq 0$. This is equivalent to say, 
$$\ddbar f \geq - Ric(\theta) - A\theta, $$
or $$f\in \textnormal{PSH}(X, Ric(\theta) + A \theta). $$
%
%
%where $\textnormal{PSH}(X, \chi) $ is the set of all quasi-plurisubharmonic functions on $X$ associated to a real valued closed $(1,1)$-form $\chi$. 

We will explain one of the motivations for  condition (\ref{assump0}) by   some examples. Let $\{E_i \}_{i=1}^I$ and $\{ F_j\}_{j=1}^J$ be two families of effective divisors of $X$.  Let $\sigma_{E_i}$, $\sigma_{F_j}$ be the defining sections for $E_i$ and $F_j$, respectively,   and $h_{E_i}$ and $h_{F_j}$  smooth hermitian metrics for the line bundles associated to $E_i$ and $F_j$ respectively. In \cite{Y1}, Yau considers the following degenerate complex Monge-Amp\`ere equations 
\begin{equation}\label{yau1}
(\theta+ \ddbar \varphi)^n =  \left( \frac{ \sum_{i=1}^I   |\sigma_{E_i}|^{2\beta_i}_{h_{E_i} } }{ \sum_{j=1}^J   |\sigma_{F_j}|^{2\alpha_j}_{h_{F_j} }  }  \right)   \theta^n,
\end{equation}
where $\alpha_j, \beta_i>0$, and various estimates are derived \cite{Y1} assuming certain bounds on the degenerate right hand side of equation (\ref{yau1}).

If we consider the following case
\begin{equation}\label{yau2}
(\theta+ \ddbar \varphi)^n =   \frac{\theta^n }{ \sum_{j=1}^J   |\sigma_{F_j}|^{2\alpha_j}_{h_{F_j} }  }   ~ .
\end{equation}
the volume measure will blow up along common zeros of $\{F_j\}_{j=1}^J$.  If the volume measure on the right hand side of the equation (\ref{yau2}) is $L^p$-integrable for some $p>1$, i.e., $$\Omega =  \left( \sum_{j=1}^J   |\sigma_{F_j}|^{2\alpha_j}_{h_{F_j} }    \right)^{-1} \theta^n$$ satisfies 
$$\frac{\Omega}{\theta^n} = \left( \sum_{j=1}^J   |\sigma_{F_j}|^{2\alpha_j}_{h_{F_j} }    \right)^{-1}  \in L^p(X, \theta^n), ~\textnormal{for~ some~} p>1,  \int_X \Omega = \int_X \theta^n, $$
then there exists a unique (up to a constant translation) continuous solution of (\ref{yau2}). Furthermore,  $\Omega$ can be approximated by smooth volume forms $\Omega_j$ (c.f. \cite{Da}) satisfying
$$Ric(\Omega_j) \geq - (A+A')\theta, ~ \left\| \frac{\Omega_j}{\theta^n} \right\|_{L^p(X, \theta^n)}\leq \left\| \frac{\Omega}{\theta^n} \right\|_{L^p(X, \theta^n)}, ~  \int_X \Omega_j = \int_X \theta^n$$
for some fixed $A'\geq 0$.
Therefore condition (\ref{assump0}) is a natural generalization of the above case. In the special case when $\{F_j\}_{j=1}^J$ is a union of smooth divisors with simple normal crossings and each $\alpha_j \in (0,1)$, the solution of equation (\ref{yau2}) has conical singularities of cone angle of $2\pi (1-\alpha_j)$ along $F_j$, $j=1, ..., J$.

We now state the first result of the paper.

\begin{theorem}\label{main1} Let $(X, \theta)$ be an K\"ahler manifold of complex dimension $n$ equipped with a K\"ahler metric $\theta$. We consider the following complex Monge-Amp\`ere equation
\begin{equation}\label{eqn:main1}(\theta + \ddbar \varphi)^n = e^{\lambda \varphi} \Omega, \end{equation}
where $\lambda =0$ or $1$, and $\Omega$ is a smooth volume form satisfying $\int_X \Omega = \int_X \theta^n$. If 
\begin{equation}\label{assump1}
  \int_X \left( \frac{\Omega}{\theta^n} \right)^p \theta^n \leq K, ~ Ric(\Omega) = -\ddbar\log \Omega \geq -A \theta,
 \end{equation}
for some  $p>1$,  $K>0$ and  $A\geq 0$, then there exists $C=C(X, \theta, p, K, A)>0$ such that the solution $\varphi$ of equation (\ref{eqn:main1}) and the K\"ahler metric $g$ associated to the K\"ahler form $\omega = \theta+ \ddbar \varphi$ satisfy the following estimates.

\smallskip

\begin{enumerate}

\item $\|\varphi - \sup_X \varphi \|_{L^\infty(X)} + \|\nabla_g \varphi\|_{L^\infty(X, g)} \leq C. $

\medskip

\item $Ric(g) \geq - C g$.

\medskip

\item $Diam(X, g) \leq C$.

\end{enumerate}

\end{theorem}

If we  write $\Omega = e^{-f} \theta^n$,   assumption (\ref{assump1}) in Theorem \ref{main1} on $\Omega$ is equivalent to the following on $f$:
\begin{equation}\label{assump2}
e^{-f}\in L^p(X, \theta), ~\int_X e^{-f}\theta = [\theta]^n, ~ f\in \textnormal{PSH}(X, Ric(\theta) + A\theta) . 
\end{equation}
 $f$ is uniformly bounded above by the plurisubharmonicity and  the K\"ahler metric $g$ associated to $\omega=\theta+ \ddbar \varphi$ is bounded below by a fixed multiple of $\theta$ (see Lemma \ref{lemma 2.2}). However, one can not expect that $g$ is bounded from above since $f$ is not uniformly bounded above as in the example of equation (\ref{yau2}). Fortunately, we can bound the diameter of $(X, g)$ uniformly by Theorem \ref{main1}. 

The gradient estimate in Theorem \ref{main1} is a generalization of the gradient estimate in \cite{S1}. The new insight in our approach is that one should estimate gradient and higher order estimates of the potential functions with respect to the new metric instead of a fixed reference metric for geometric complex Monge-Amp\`ere equations such as those studied in Theorem \ref{main1}. We refer interested readers to \cite{B, PS} for gradient estimates for complex Monge-Amp\`ere equations with respect to various background metrics.

Let $\mathcal{M}(X, \theta, p, K, A)$ be the space of all solutions of equation (\ref{eqn:main1}), where $\Omega$ satisfies   assumption (\ref{assump1}) in Theorem \ref{main1}.  We also identify $\mathcal{M}(X, \theta, p, K, A)$ with the space of K\"ahler forms $\omega =\theta+ \ddbar \varphi$ for $\varphi\in \mathcal{M}(X, \theta, p, K, A)$. An immediate consequence of Theorem \ref{main1} is a uniform noncollapsing condition for  $\mathcal{M}(X, \theta, p, K, A)$. More precisely,  there exists a constant $C=C(X, \theta, p, K, A)>0$ such that for all K\"ahler metric $g$ associated to $\omega  \in \mathcal{M}(X, \theta, p, K, A)$ and for any point $x\in X$, $0< r <1$, 
\begin{equation} \label{noncol}
C^{-1} r^{2n} \leq Vol_g(B_g(x, r)) \leq Cr^{2n} ,
\end{equation}
where $B_g(x, r)$ is the geodesic ball centered at $q$ with radius $r$ in $(X, g)$. 

Combining the lower bound of Ricci curvature and the non-collapsing condition  (\ref{noncol}), we can apply the theory of degeneration of Riemannian manifolds \cite{CC} so that  any sequence of K\"ahler manifolds $(X, g_j) \in \mathcal{M}(X, \theta, p, K, A)$, after passing to a subsequence,  converges to a compact metric space $(X_\infty, d_\infty)$ with well-defined tangent cones of Hausdorff dimension $2n$ at each point in $X_\infty$.  In the case of equation (\ref{yau2}), we believe the solution induces a unique Riemannian metric space homeomorphic to the original manifold $X$ and all tangent cones are unique and   biholomorphic to $\mathbb{C}^n$. If this is true, one might even be able to establish higher order expansions for the solution. The ultimate goal of our approach is to construct canonical domains and equations on the blow-up of  solutions for geometric  degenerate complex Monge-Amp\`ere equations, by degeneration of Riemannian manifolds. 

We also remark that if we replace the lower bound for $Ric(\Omega)$ by an upper bound 
$$Ric(\Omega) \leq A \theta $$ 
in  assumption \eqref{assump1} of Theorem \ref{main1},  we can still obtain a uniform diameter upper bound. This in fact easily follows from the argument for the second order estimates of Yau \cite{Y1} and Aubin \cite{A}.

We will also use similar techniques in the proof of Theorem \ref{main1} to obtain diameter estimates in more geometric settings. Before that, let us introduce a few necessary and well-known notions in  complex geometry.

\begin{defn} Let $X$ be a K\"ahler manifold of complex dimension $n$ and $\alpha \in H^2(X, \mathbb{R})\cap H^{1, 1}(X, \mathbb{R})$ be nef. The numerical dimension of the class $\alpha$ is given by 
\begin{equation}
\nu(\alpha) = \max \{ k=0, 1, ..., n ~|~ \alpha^k \neq 0 ~\textnormal{in}~ H^{2k}(X, \mathbb{R})  \}.
\end{equation}
when $\nu(\alpha) = n$, the class $\alpha$ is said to be big.

\end{defn}

The numerical dimension $\nu(\alpha)$ is always no greater than $\dim_{\mathbb{C}}(X)$. 

\begin{defn} Let $X$ be a K\"ahler manifold and $\alpha \in H^2(X, \mathbb{R})\cap H^{1, 1}(X, \mathbb{R})$.  Then the class $\alpha$ is nef if  $\alpha+ \mathcal{A}$ is a K\"ahler class for any K\"ahler class $\mathcal{A}$. 

\end{defn}

When the canonical bundle $K_X$ is nef, $X$ is said to be a minimal model.    The abundance conjecture in birational geometry predicts that the canonical line bundle is always semi-ample (i.e. a sufficiently large power of the canonical line bundle is globally generated) if it is nef.

\begin{defn} \label{extreme} Let $\vartheta$ be a smooth real valued closed $(1,1)$-form on a K\"ahler manifold $X$. The extremal function $V$ associated to the form $\vartheta$ is defined by
$$V(z) =  \sup\{ \phi(z) ~|~ \vartheta+ \ddbar \phi \geq 0, ~\sup_X \phi = 0\} ,  $$ 
for all $z\in X$.

\end{defn}

Any $\psi \in \textnormal{PSH}(X, \vartheta)$ is said to have minimal singularities defined by Demailly (c.f. \cite{BD}) if $\psi- V$ is bounded.

 Let $(X,\theta)$ be a  K\"ahler manifold of complex dimension $n$ equipped with a K\"ahler metric $\theta$. Suppose $\chi$ is a real valued smooth closed $(1,1)$-form and its class $[\chi]$ is  nef and of numerical dimension $\kappa$. 
We consider the following family of complex Monge-Amp\`ere equations
\begin{equation}\label{eqn:main2}(\chi + t \theta + \ddbar \varphi_t)^n = t^{n-\kappa} e^{\lambda \varphi_t  + c_t} \Omega, ~ \textnormal{for} ~t\in (0,1],  \end{equation}
where $\lambda=0$, or $1$, and $c_t$ is a normalizing constant such that 
\begin{equation}\label{norm2}
\int_X t^{n-\kappa} e^{ c_t} \Omega = \int_X (\chi+ t \theta)^n. 
\end{equation}
Straightforward calculations show that $c_t$ is uniformly bounded for $t \in (0, 1]$.
The following proposition generalizes the result in \cite{BEGZ, K, EGZ1, ZZ} by studying a family of collapsing complex Monge-Amp\`ere equations. It also generalizes the results in \cite{DP, EGZ, KT} for the case when the limiting reference form is semi-positive. 

\begin{prop} \label{main2} We consider equation (\ref{eqn:main2}) with the normalization condition (\ref{norm2}). Suppose the volume measure $\Omega$ satisfies
$$\int_X \left( \frac{\Omega}{\theta^n}  \right)^p \theta^n \leq K$$   
for some $p>1$ and $K>0$. 
Then there exists a unique $\varphi_t\in \textnormal{PSH}(X,\chi + t \theta) $ up to a constant translation solving  equation (\ref{eqn:main2}) for all $t\in (0, 1]$. Furthermore, there exists  $C=C(X, \chi, \theta, p, K)>0$ such that for all $t\in (0, 1]$,
$$\|(\varphi_t - \sup_X \varphi_t) - V_t \|_{L^\infty(X)} \leq C, $$
where $V_t$ is the extremal function associated to $\chi+t\theta$ as in Definition \ref{extreme}. 
\end{prop}

Proposition \ref{main2} can be applied to generalize Theorem \ref{main1}, especially for minimal K\"ahler manifolds in a geometric setting. A K\"ahler manifold is called a minimal model if its canonical bundle is nef.

\begin{theorem} \label{main3} Suppose $X$ is a smooth minimal model equipped with a smooth K\"ahler form $\theta$. For any $t> 0$, there exists a unique smooth twisted K\"ahler-Einstein metric $g_t$ on $X$ satisfying
\begin{equation}\label{eqn:tKE}Ric(g_t) = - g_t + t \theta . \end{equation}
There exists $C=C(X, \theta)>0$ such that for all $t\in (0,1]$, 
$$\diam(X, g_t) \leq C. $$
Furthermore, for any $t_j \rightarrow 0$, after passing to a subsequence,  the twisted K\"ahler-Einstein manifolds $(X, g_{t_{j}})$ converge  in Gromov-Hausdorff topology to a compact metric length  space $(\mathbf{Z}, d_\mathbf{Z})$.
The K\"ahler forms $\omega_{t_j} $ associated to $g_{t_j}$  converge in  distribution to a  nonnegative closed current $\widetilde\omega = \chi + \ddbar \widetilde \varphi $ for some $\widetilde\varphi  \in \textnormal{PSH}(X, \chi)$ of  minimal singularities, where $\chi\in [ K_X]$ is a fixed smooth closed $(1,1)$-form.

\end{theorem}

Both Theorem \ref{main1} and Theorem \ref{main3} are generalization and improvement for the techniques developed in \cite{S1} for diameter and distance estimates. With the additional bounds on the volume measure, we transform Kolodziej's analytic $L^\infty$-estimate to a geometric diameter estimate. It is a natural question to ask how the metric space $(\mathbf{Z}, d_\mathbf{Z})$ is related to the current $\widetilde \omega$ on $X$. We conjecture $\widetilde \omega$ is smooth on an open dense set of $X$ and its metric completion coincides with  $(\mathbf{Z}, d_\mathbf{Z})$. However, at this moment, we do not even know the Hausdorff dimension or uniqueness of $(\mathbf{Z}, d_\mathbf{Z})$.

When $X$ is a minimal model of general type, Theorem \ref{main3} is proved in \cite{S1, S2} and the result in \cite{TW} shows that the singular set is closed and of Hausdorff dimension no greater than $2n-4$. 

We can also replace the smooth K\"ahler form $\theta$ in Theorem \ref{main3} by Dirac measures along effective divisors. For example, if $\{E_j\}_{j=1}^J$ is a union of smooth divisors with normal crossings and
$$\sum_{j=1}^J a_j E_j$$
is an ample  $\mathbb{Q}$-divisor  with some $a_j \in (0, 1)$ for $j=1, ..., J$. Then Theorem \ref{main3} also holds if we let $\theta = \sum_{j=1}^J a_j [E_j]. $ In this case, the metric $g_t$ is a conical K\"ahler-Einstein metric with cone angles of $2\pi (1-a_j)$ along each complex hypersurface $E_j$.

 A special case of the abundance conjecture is proved by Kawamata \cite{Ka} for minimal models of general type. When $X$ is a smooth minimal model of general type, it is  recently proved by the third named author \cite{S2} that the limiting metric space $(\mathbf{Z}, d_\mathbf{Z})$ in Theorem \ref{main3} is unique and is  homeomorphic to the algebraic canonical model $X_{can}$ of $X$. This gives an analytic proof of Kawamata's result using complex Monge-Amp\`ere equations, Riemannian geometry and geometric $L^2$-estimates.   Theorem \ref{main3} also provides a Riemannian geometric model for the non-general type case. This analytic approach will shed light  on the abundance conjecture if such a metric model is unique with reasonably good understanding of its tangle cones.  

Theorem \ref{main3} can also be easily generalized to a Calabi-Yau manifold $X$ equipped with a nef line bundle $L$ over $X$ of $\nu(L)=\kappa$.  

Our final result assumes semi-ampleness for the canonical line bundle and aims to connect the algebraic canonical models to geometric canonical models. Let $X$ be a K\"ahler manifold of complex dimension $n$. If the canonical bundle $K_X$ is semi-ample, the pluricanonical system induces a holomorphic surjective map
$$\Phi: X \rightarrow X_{can}$$
from $X$ to its unique canonical model $X_{can}$. In particular, $\dim_\mathbb{C} X_{can} = \nu(X)$.  We let    
$\mathbf{S}$ be the set of all singular fibers of $\Phi$ and $\Phi^{-1} \left( S_{X_{can}} \right)$, where $S_{X_{can}}$ is the singular set of $X_{can}$. The general fibre of $\Phi$ is a smooth Calabi-Yau manifold of complex dimension $n-\nu(X)$. It is proved in \cite{S1, S2} that there exists a unique twisted K\"ahler-Einstein current $\omega_{can}$ on $X_{can}$ satisfying
\begin{equation}\label{twke}
Ric(\omega_{can}) = - \omega_{can} + \omega_{WP}, 
\end{equation}
where  $\Phi^* \omega_{can} \in -c_1(X)$ and $\omega_{WP}$ is the Weil-Petersson metric for the variation of the Calabi-Yau fibres. In particular,  $\omega_{can}$ has bounded local potentials and is smooth on $X_{can}\setminus{S}_{can}$. We let $g_{can}$ be the smooth K\"ahler metric associated to $\omega_{can}$ on  $X_{can}\setminus {S}_{can}$.

\begin{theorem} \label{main4} Suppose $X$ is a projective manifold of complex dimension $n$ equipped with a K\"ahler metric $\theta$. If the canonical bundle $K_X$ is semi-ample and $\nu(K_X)=\kappa\in\mathbb N$, then the following hold for the twisted K\"ahler-Einstein metrics $g_t$ satisfying 
$$Ric(g_t) = - g_t + t \theta, ~ t\in (0, 1]. $$

\begin{enumerate}[label=(\arabic*)]

\item There exists $C>0$ such that for all $t\in (0, 1]$, 
$$\diam(X, g_t) \leq C. $$

\item Let $\omega_t$ be the K\"ahler form associated to $g_t$. For any compact subset $K\subset\subset X\backslash \mathbf{S}$, we have 
\begin{equation*}%\label{eqn:C0}
\| g_t - \Phi^*g_{can}\|_{C^0(K, \theta)}\to 0,\quad \text{as }t\to 0. 
\end{equation*}
%
%\medskip

\item The rescaled metrics $t^{-1}\omega_t|_{X_y}$ converge uniformly to a Ricci-flat K\"ahler metric $\omega_{CY, y}$ on the fibre $X_y= \Phi^{-1}(y)$ for any $y\in  X_{can} \setminus \Phi(\mathbf{S})$, as $t\rightarrow 0$. 

\medskip

\item For any sequence $t_j \rightarrow 0$,  after passing to a subsequence,  $(X, g_{t_{j}})$ converge in Gromov-Hausdorff topology to a compact metric space $(\mathbf{Z}, d_\mathbf{Z})$. Furthermore, $X_{can}\setminus {S}_{can}$ is embedded as an open subset in the regular part $\mathcal{R}_{2\kappa}$ of $(\mathbf{Z}, d_\mathbf{Z})$ and $(X_{can} \setminus {S}_{can}, \omega_{can})$ is locally isometric to its open image. 

%\item The rescaled metrics $(1+t)\omega_t|_{X_y}$ converge uniformly to the semi-flat metric $\omega_{SF}|_{X_y}$ for any $y\in K'\Subset X_{can}^\circ$, and $X_y = \pi^{-1}(y)$ is the fiber over $y$.
\smallskip

\end{enumerate}
In particular, if $\kappa =1$, $( \mathbf{Z}, d_\mathbf{Z} )$ is homeomorphic to $X_{can}$, with the regular part being open and dense, and each tangent cone being a metric cone on $\mathbb{C}$ with cone angle less than or equal to $2\pi$. 

\end{theorem}

We remark that a special case of Theorem \ref{main4} is proved in \cite{ZZh} with a different approach for  $\dim_{\mathbb{C}} X=2$.  In general, the collapsing theory in Riemannian geometry has not been fully developed except in lower dimensions. In the K\"ahler case, one hopes the rigidity properties can help us understand the collapsing behavior  for  K\"ahler metrics of Einstien type as well as long time solutions of the K\"ahler-Ricci flow on algebraic minimal models. Key analytic and geometric estimates in the proof of (2) in Theorem \ref{main4} are established in \cite{ST1, ST2} for  the collapsing long time solutions of the K\"ahler-Ricci flow and its elliptic analogues. The proof for (3) and (4) is a technical modification of  various local results of \cite{To, GTZ, TWY, GTZ1}, where collapsing behavior for families of Ricci-flat Calabi-Yau manifolds is comprehensively studied. Theorem \ref{main4} should also hold for K\"ahler manifolds with some additional arguments.

Finally, we will also apply our method to a continuity scheme proposed in \cite{LT} to study singularities arising from contraction of projective manifolds. This is an alternative approach for the analytic minimal model program developed in \cite{ST1, ST2, ST3} to understand birational transformations via analytic and geometric methods (see also \cite{SW1, SW2, SY, S0}). Compared to the K\"ahler-Ricci flow, such a scheme has the advantage of prescribed Ricci lower bounds and so one can apply the Cheeger-Colding theory for degeneration of Riemannian manifolds, on the other hand, it loses the canonical soliton structure for the analytic transition of singularities corresponding to birational surgeries such as flips.  

Let $X$ be a projective manifold of complex dimension $n$. We choose an ample line bundle $L$ on $X$ and we can assume that $L-K_X$ is ample, otherwise we can replace $L$ by a sufficiently large power of $L$. We choose $\theta \in [L-K_X]$ to be a smooth K\"ahler form and consider the following curvature equation
\begin{equation}\label{curveq}
Ric(g_t) = - g_t + t \theta, ~~ t \in [0, 1]. 
\end{equation}
Let 
\begin{equation}\label{tmax}
t_{min} = \inf\{ t \in [0, 1]~|~ \textnormal{ equation~ (\ref{curveq})~is solvable~ at~} t\}.
\end{equation}
It is straightforward to verify that $t_{min} <1$ by the usual continuity method (c.f. \cite{LT}). The goal is to solve equation (\ref{curveq}) for all $t\in (0, 1]$, however, one might have to stop at $t= t_{min}$ when $K_X$ is not nef.

\begin{theorem}\label{main5}
Let $g_t$ the solution of equation (\ref{curveq}) for $t\in (t_{min}, 1]$. There exists $C=C(X, \theta)>0$ such that for any $t\in (t_{min}, 1]$, 
\begin{equation}
Diam(X, g_t) \leq C. 
\end{equation}

\end{theorem}

Theorem \ref{main3} is a special case of Theorem \ref{main5} when $t_{min}= 0$ (c.f. \cite{S1}). When $t_{min}>0$, Theorem \ref{main5} is also proved in \cite{LTZ} with the additional assumption that  $$t_{min} L +  (1-t_{min})K_X$$ is semi-ample and big. The diameter estimate immediately allows one to identify the geometric limit as a compact metric length space when $t\rightarrow t_{min}$. In particular, it is shown in \cite{LTZ} that the limiting metric space is homeomorphic to the projective variety from the contraction induced by the $\mathbb{Q}$-line bundle $t_{min}  L + (1-t_{min})K_X $ when it is big and semi-ample. One can also use Theorem \ref{main5} to obtain a weaker version of Kawamata's base point free theorem in the minimal model theory (c.f. \cite{GS}).
  If $t_{min} L +  (1-t_{min})K_X$ is not big, our diameter estimate still holds and we conjecture the limiting collapsed metric space of $(X, g_t)$ as $t\rightarrow t_{min}$ is unique and is homeomorphic a lower dimensional projective variety from the contraction induced by $t_{min}  L + (1-t_{min})K_X $.

\section{Proof of Theorem \ref{main1} } Throughout this section, we let $\varphi\in \textnormal{PSH}(X,\theta)$ be the solution of the equation \eqref{eqn:main1} satisfying the condition (\ref{assump1}) in Theorem \ref{main1}.
We let $\omega= \chi+ \ddbar \varphi$ and let $g$ be the K\"ahler metric associated to $\omega$. 

\begin{lemma} \label{lemma 2.1}There exists $C = C(X,\theta, p, K)>0$ such that $$\|\varphi -\sup_X \varphi \|_{L^\infty(X)} \leq C. $$

\end{lemma}

\begin{proof} The $L^\infty$ estimate immediately follows from Kolodziej's theorem \cite{K}.

\end{proof}

The following is a result similar to Schwarz lemma.

\begin{lemma} \label{lemma 2.2}There exists $C=C(X, \theta, p, K, A)>0$ such that 
$$\omega \geq C \theta. $$

\end{lemma}

\begin{proof}  There exists $C=C(X, \theta, A)>0$ such that 
$$\Delta_\omega \log \tr_\omega (\theta) \geq -C - C \tr_\omega (\theta), $$
where $\Delta_\omega$ is the Laplace operator associated with $\omega$.
Then let 
$$H = \log \tr_\omega (\theta) - B (\varphi -\sup_X \varphi)$$ 
for some $B >2 C$. Then
$$\Delta_\omega H \geq C \tr_\omega (\theta ) - C. $$ It follows from maximum principle and the $L^\infty$-estimate in Lemma \ref{lemma 2.1} that $\sup_X \tr_\omega \theta \le C$.

\end{proof}

Lemma \ref{lemma 2.2} immediately gives the uniform Ricci lower bound.

\begin{lemma} There exists $C=(X, \theta, p, K, A)>0$ such that $$Ric(g) \geq - C g. $$

\end{lemma} 

\begin{proof} We calculate
$$Ric(g) = -\lambda g + Ric(\Omega) + \lambda \theta \geq - \lambda g - (A-\lambda) \theta\geq - C g$$
for some fixed constant $C>0$ by  Lemma \ref{lemma 2.2}.

\end{proof}

We will now prove the uniform diameter bound. 

\begin{lemma} There exists $C=(X, \theta, p, K, A)>0$ such that $$Diam(X, g) \leq C. $$

\end{lemma}

\begin{proof} We first fix a sufficiently small $\epsilon=\epsilon(p)>0$ so that $p- \epsilon>1$. Suppose $Diam(X, g) = D$ for some $D\geq 4$. Let $\gamma:[0,D]\to X$ be a normal minimal geodesic with respect to the metric $g$ and choose the points $\{x_i = \gamma(6i)\}_{i=0}^{[D/6]}$. It is clear that the balls $\{B_{g}(x_i,3)\}_{i=0}^{[D/6]}$ are  disjoint, so
$$\sum_{i=0}^{[D/6]} \vol_\theta\xk{B_g(x_i,3)   }\le \int_X \theta^n = V,$$
hence there exists a geodesic ball $B_g(x_i, 3)$ such that $$\vol_\theta\big(B_g(x_i, 3)\big) \leq 6V D^{-1}.$$ We fix such $x_i$ and construct a cut-off function $\eta(x) = \rho(r(x))  \geq 0$ with 
$$r(x) = d_g(x,x_i)$$ such that 
$$\eta=1~\textnormal{on}~ B_g(x_i, 1),  ~ \eta=0 ~\textnormal{outside}~B_g(x_i, 2), $$
and
$$ \rho\in [0,1] ,\quad \rho^{-1} (\rho')^2\le C(n), \quad |\rho''| \leq C(n). $$
Define a function $F>0$ on $X$ such that $F=1$ outside $B_g(x_i, 3)$, $F= D^{\frac{\epsilon}{p(p-\epsilon)}}$ on $B_g(x_i, 2)$, and 
$$\int_X F \Omega = [\theta]^n, ~~ \int_X \Big(\frac{F\Omega}{\theta^n} \Big)^{p-\epsilon} \theta^n \leq \Big( \int_X F^{\frac{p(p-\epsilon)}{\epsilon} } \theta^n\Big)^{\frac{\epsilon}{p}}\Big( \int_X \big(\frac{\Omega}{\theta^n} \big)^p \theta^n \Big)^{ \frac{p-\epsilon}{p}} \leq C$$
for some $C=C(X, \theta, p, K)>0$.

We now consider the equation
$$(\theta + \ddbar \phi)^n = e^{\lambda \phi} F \Omega. $$
By similar argument as before, $\| \phi - \sup_X \phi\|_{L^\infty} \leq C = C(X, \theta, p, K)$. Let $\hat g= \theta + \ddbar \phi$. Then on $B_g(x_i, 2)$, 
$$Ric(\hat g) = - \lambda \hat g + Ric(\Omega) + \lambda \theta, \quad Ric(g) = - \lambda g + Ric(\Omega) + \lambda \theta. $$
In particular, 
$$ \Delta_g \log \frac{\hat\omega^n} {\omega^n} = -\lambda n + \lambda \tr_g(\hat g) , $$
where $\Delta_g = \Delta_\omega$. Let 
$$H = \eta \Big( \log \frac{\hat\omega^n} {\omega^n}  -  \xk{(\varphi - \sup_X \varphi )- (\phi - \sup_X \phi) } \Big). $$  
On $B_g(x_i, 2), $ we have
$$ \Delta_g  H = -(\lambda +1) n + (\lambda+1)  \tr_g(\hat g) \geq -2n + n \left( \frac{\hat\omega^n} {\omega^n} \right)^{1/n} .$$
In general,  on the support of $\eta$, we have 
\begin{eqnarray*}
\Delta_g H &\geq& \eta \left(-2n + n \Big(  \frac{\hat \omega^n} {\omega^n} \Big)^{1/n} \right) + 2 \eta^{-1} Re\left(  \nabla H \cdot  \nabla \eta \right) - 2 \frac{ H  |\nabla \eta|^2}{\eta^2} + \eta^{-1}H \Delta_g \eta \\
&\geq &  \eta^{-1}\bk{C \eta^2 e^{H/(n\eta)} + 2 Re\big(  \nabla H \cdot  \nabla \eta \big) - 2 \frac{ H  |\nabla \eta|^2}{\eta} + H \Delta_g \eta - 2n \eta^2}.
%
%& \geq & C \eta \left( e^{H/\eta} - C^2( H/\eta)  - 1 \right)+ 2 \eta^{-1} Re\left(  \nabla H \cdot  \nabla \eta \right).
%
%
\end{eqnarray*}
We may assume $\sup_X H>0$, otherwise we already have upper bound of $H$. The maximum of $H$ must lie at $B_g(x_i,2)$ and at this point $$\Delta_g H\le 0 ,\quad \abs{\nabla H} = 0.$$ By Laplacian comparison we have
$$\Delta_g \eta = \rho' \Delta r + \rho'' \ge -C,\quad \frac{\abs{\nabla \eta}}{\eta} = \frac{(\rho')^2}{\rho}\le C.$$ 
 So at the maximum of $H$, it holds that
$$0\ge C \eta^2 e^{H/(n\eta)} - C H - 2n\ge C H^2 - C H - 2n, $$ therefore $\sup_X H\le C$. In particular on the ball $B_g(x_i,1)$ where $\eta\equiv 1$, it follows that $\frac{\hat \omega^n}{\omega^n}\le C$. From the definition of $\hat\omega$ and $\omega$, 
$$C\ge  \frac{\hat \omega^n}{\omega^n} = D^{\frac{\epsilon}{p (p-\epsilon)}} e^{\lambda (\phi - \varphi)}.$$Combined with the $L^\infty$-estimate of $\phi$ and $\varphi$, we conclude that 
$$D\le C = C(n,p,\theta,A, K).$$

\end{proof}

\begin{lemma} There exists $C=(X, \theta, p, K, A)>0$ such that $$ \sup_X |\nabla_g \varphi|_g \leq C. $$

\end{lemma}

\begin{proof} Straightforward calculations show that 
$$\Delta_g \varphi = n - \tr_g (\theta), $$
\begin{eqnarray*}
 \Delta_g |\nabla \varphi|^2_g &=& |\nabla \nabla \varphi|^2 + |\nabla \bar{\nabla} \varphi|^2 + g^{i\bar l} g^{k \bar j} R_{i\bar j} \varphi_k \varphi_{\bar l} - 2\nabla \varphi \cdot \nabla \tr_g(\theta) \\
 &\geq & |\nabla \nabla \varphi|^2 + |\nabla \bar{\nabla} \varphi|^2  - C |\nabla \varphi|^2 - 2 \nabla \varphi \cdot \nabla \tr_g(\theta),
\end{eqnarray*}
and
$$\Delta_g \tr_g\theta = \tr_g \theta\cdot \Delta_g\log\tr_g\theta + \frac{\abs{\nabla \tr_g\theta}}{\tr_g\theta}\ge -C + c_0\abs{\nabla \tr_g \theta}$$
for some uniform constant $c_0, C>0$. 
We choose constants $\alpha$ and $B$ satisfying 
$$\alpha > 4 c_0^{-1}>4, ~B>\sup_X \varphi + 1$$
 and define 
 $$H = \frac{\abs{\nabla \varphi}}{B - \varphi} + \alpha \tr_g \theta.$$ 
 Then we have
\begin{equation}\label{eqn:grad 1}\begin{split}
\Delta H & \ge \frac{\abs{\nabla \nabla \varphi} + \abs{\nabla \bar \nabla \varphi}}{B-\varphi} - C \frac{\abs{\nabla \varphi}}{B-\varphi}- \frac{\abs{\nabla \varphi} (\tr_g \theta - n)}{(B-\varphi)^2} - 2 (1+\alpha)\frac{\innpro{ \nabla \varphi, \nabla \tr_g\theta  }}{B-\varphi}\\
& \quad -\alpha C + \alpha c_0\abs{\nabla \tr_g\theta} + 2\innpro{\frac{\nabla \varphi}{B-\varphi},\nabla H}.
\end{split}\end{equation}
We may assume at the maximum point $z_{max}$ of $H$, $|\nabla \varphi|>\alpha$ and $H>0$, otherwise we are done. At $z_{max}$,
$$\nabla H = 0, ~\Delta H \le 0$$
and so at $z_{max}$
\begin{align*}
\nabla |\nabla \varphi| = \frac 1 2 \bk{ - H \frac{\nabla \varphi}{|\nabla \varphi|} - \alpha(B-\varphi) \frac{\nabla \tr_g \theta}{|\nabla \varphi|} + \alpha \frac{\tr_g \theta \nabla \varphi}{|\nabla \varphi|} }.
\end{align*} 
By Kato's inequality $\abs{ \nabla |\nabla \varphi|  }\le \frac{\abs{\nabla \nabla \varphi} + \abs{\nabla \bar \nabla \varphi}}{2}$, it follows that
\begin{equation}\label{eqn:grad 2}\begin{split}
\frac{\abs{\nabla \nabla \varphi} + \abs{\nabla \bar \nabla \varphi}}{B-\varphi} & \ge \frac 1{ 2(B-\varphi)} \bk{ H^2 + \alpha^2 (B-\varphi)^2 (\tr_g \theta)^2 + \alpha^2 \frac{\abs{\nabla \tr_g \theta}}{\abs{\nabla \varphi}} \\
&\quad - 2 \alpha H (B-\varphi) \tr_g \theta 
 - 2 \alpha H \frac{|\nabla \tr_g \theta|}{|\nabla \varphi|} - 2 \alpha^2 (B-\varphi) \tr_g \theta \frac{|\nabla\tr_g \theta|}{|\nabla \varphi|}   } \\
 &\ge \frac{H^2}{4 (B-\varphi)}- C H -\frac{\abs{\nabla \tr_g \theta}}{B-\varphi} - C |\nabla \tr_g\theta|
\end{split}\end{equation}
for some uniform constant $C>0$.
After substituting \eqref{eqn:grad 2} to \eqref{eqn:grad 1} and applying Cauchy-Schwarz inequality, we have at $z_{max}$ 
\begin{align*}0\ge& \frac{H^2}{4(B-\varphi)} - CH -C - \frac{2\abs{\nabla \tr_g\theta}}{B-\varphi} - C |\nabla \tr_g\theta| + 4 \abs{\nabla \tr_g\theta}\\
\ge & \frac{H^2}{4(B-\varphi)} - CH -C,\end{align*}
for some uniform constant $C>0$. Therefore $\max_X H\le C$ for some $C=C(X,\theta,\Omega,A,p,K)$.   The lemma then immediately follows from  Lemma \ref{lemma 2.1} and Lemma \ref{lemma 2.2}.

\end{proof}

\section{ Proof of Proposition \ref{main2} }

In this section, we will prove Proposition \ref{main2} by applying the techniques in \cite{BEGZ, K, EGZ1, DP}. 

Let $X$ be a K\"ahler manifold of dimension $n$. Suppose $\alpha$ is nef class on $X$ of numerical dimension $\kappa\geq 0$.  Let  $\chi \in \alpha$ be a smooth closed $(1,1)$-form. We define the extremal function $V_\chi$ by
\begin{equation}
V_\chi =   \sup\{ \phi ~|~\chi+\ddbar \phi \geq 0, ~ \phi \leq 0\}  .
\end{equation}
Let $\theta$ be a fixed smooth K\"ahler metric on $X$. Then we define the perturbed extremal function $V_t$ for $t\in (0, 1]$ by
\begin{equation}
V_t =  \sup\{ \phi ~|~\chi+ t \theta+\ddbar \phi \geq 0, ~ \phi \leq 0\} .
\end{equation}
The above extremal functions were introduced in \cite{BEGZ} when $\alpha$ is big.  

We first rewrite the equation (\ref{eqn:main2}) for $\lambda=0$ as  follows 
\begin{equation}\label{eqn:3.11}(\chi + t\theta + \ddbar \varphi_t)^n = t^{n-\kappa} e^{-f+ c_t} \theta^n,\quad \sup_X \varphi_t = 0, \end{equation}
$t\in (0, 1]$ by letting $\Omega = e^{-f}\theta^n$, where  $c_t$ is the normalizing constant satisfying 
$$t^{n-\kappa} \int_X e^{-f+c_t} \theta^n = \int_X (\chi + t \theta)^n . $$  
$f$ satisfies the following uniform bound
$$\int_X e^{-pf} \theta^n \leq K, $$
for some $p>1$ and $K>0$.

%\begin{proof}

%\end{proof}

The following definition is an extension of the capacity introduced in  \cite{K, EGZ1, DP, BEGZ}.

\begin{defn} We define the capacity $Cap_{\chi_t}( \mathcal{K})$ for a  subset $  \mathcal{K} \subset X$ by
\begin{equation}
Cap_{\chi_t}( \mathcal{K}) = \sup\Big\{ \int_ \mathcal{K} (\chi_t + \ddbar u)^n ~|~ u\in \textnormal{PSH}(X, \chi_t), ~ 0\leq u - V_t \leq 1 \Big\},
\end{equation}where $\chi_t = \chi + t \theta$ is the reference metric in \eqref{eqn:3.11}.  We also define the extremal function $V_{t,  \mathcal{K}}$ by
\begin{equation}
V_{t, \mathcal{K}} = \sup \left\{ u \in \textnormal{PSH}(X, \chi_t)~|~u \leq 0, ~\textnormal{on} ~  \mathcal{K} \right\} .
\end{equation}

\end{defn}

If $K$ is open, then we have 
\begin{enumerate}
\item $V_{t,\mathcal{K}}  \in \textnormal{PSH}(X, \chi_t)\cap L^\infty(X)$, 
\medskip
\item $(\chi_t + \ddbar V_{t,\mathcal{K}} )^n = 0$ on $X\setminus \overline{ \mathcal{K}}$.  

\end{enumerate}

\begin{lemma}\label{lemma 3.1}Let $\varphi_t$ be the solution to \eqref{eqn:3.11}. Then there exist $\delta= \delta(X, \chi, \theta)>0$ and $C=C(X, \chi, \theta, p, K)>0$ such that for any open set $ \mathcal{K}\subset X$ and $t\in (0, 1]$,
\begin{equation} \label{capm}
\frac{1}{[\chi_t^n]}\int_ \mathcal{K} (\chi_t+\ddbar\varphi_t)^n \leq C e^{-\delta \left( \frac{[\chi_t^n]}{Cap_{\chi_t}(K) }\right)^{\frac{1}{n}}}.
\end{equation} 
\end{lemma}
\begin{proof}
Since $[\chi^m] = 0$ for $\kappa +1\le m\le n$
\begin{align*}
[\chi_t^n] = \int_X \chi_t^n = \int_X \sum_{k=0}^n\binom{n}{k}\chi^k\wedge t^{n-k} \theta^{n-k} = \int_X \sum_{k=0}^\kappa \binom{n}{k}\chi^k\wedge t^{n-k} \theta^{n-k} = O(t^{n-\kappa}).
\end{align*}
It follows that the normalizing constant $c_t$ in \eqref{eqn:3.11} is uniform bounded. Let $M_{t,\mathcal{K}} = \sup_X V_{t,\mathcal{K}}$. Then we have
\begin{align*}
\frac{1}{[\chi_t^n]} \int_{\mathcal K}(\chi_t + \ddbar \varphi_t)^n & = \frac{t^{n-\kappa} e^{c_t}}{[\chi_t^n]}\int_K e^{-f} \theta^n\\
&\le \frac{t^{n-\kappa} e^{c_t}}{[\chi_t^n]} \int_{\mathcal K} e^{-f} e^{-\delta V_{t,\mathcal{K}}/q} \theta^n,  \quad\text{since $V_{t,\mathcal{K}}\le 0$ on $K$}\\
& \le  \frac{t^{n-\kappa} e^{c_t}}{[\chi_t^n]} e^{-\delta M_{t,\mathcal{K}}/q} \int_X e^{-f} e^{-\delta(V_{t,\mathcal{K}} - M_{t,\mathcal{K}})/q} \theta^n\\
&\le \frac{t^{n-\kappa} e^{c_t}}{[\chi_t^n]} e^{-\delta M_{t,\mathcal{K}}/q} \bk{\int_X e^{-p f}\theta^n}^{1/p} \bk{ \int_X e^{-\delta (V_{t,\mathcal{K}} - M_{t,\mathcal{K}}) }\theta^n  }^{1/q}\\
&\le C e^{-\delta M_{t,\mathcal{K}}/q },
\end{align*}
where $\frac{1}{p} + \frac 1 q = 1$.
Obviously,  there exists $\gamma=\gamma (X, \chi, \theta)>0$ such that for all $t\in (0, 1]$, 
$$V_{t,\mathcal{K}}\in \textnormal{PSH}(X,\gamma \theta). $$ We apply the global  H\"ormander's estimate (\cite{T1}) so that there exists $\delta=\delta(X, \chi, \theta)>0$ such that 
$$ \int_X e^{-\delta (V_{t,\mathcal{K}} - \sup_{X} V_{t,\mathcal{K}})}\theta^n\le C_\delta.$$
 To complete the proof, it suffices to to show 
\begin{equation}\label{eqn:MtK}
M_{t,\mathcal{K}} + 1\ge \bk{ \frac{[\chi_t^n]}{Cap_{\chi_t}(  \mathcal{K})}  }^{1/n}. 
\end{equation}
First we observe that by definition
$$\sup_X \bk{ (V_{t,\mathcal{K}} - \sup_X V_{t,\mathcal{K}}) - V_t } \leq 0, $$
since $V_{t,\mathcal{K}} -\sup_X V_{t,\mathcal{K}}\in \textnormal{PSH}(X, \chi_t)$ is nonpositive. On the other hand, $V_{t,\mathcal{K}} \geq V_t$. 
This immediately implies that 
\begin{equation}\label{crt1}
0\leq  V_{t,\mathcal{K}} - V_t \leq \sup_X V_{t,\mathcal{K}}=M_{t,\mathcal{K}} . 
\end{equation}
We break the rest of the proof into two cases. 

\medskip

\noindent{ $\bullet$ \bf{The case when} $M_{t, K} > 1$.} We let 
$$\psi_{t, \mathcal{K}} =M_{t,\mathcal{K}}^{-1}(V_{t,\mathcal{K}} - V_t) + V_t.$$ 
Then
$$V_t \leq \psi_{t, \mathcal{K}} \leq V_t +1$$
and by \eqref{crt1}. \begin{equation}\label{eqn:3.17}\begin{split}
\frac{1}{M_{t,\mathcal{K}}^n}&= \frac{1}{M_{t,\mathcal{K}}^n} \frac{\int_X (\chi_t + \ddbar V_{t,\mathcal{K}})^n}{[\chi_t^n]} =\frac{1}{[\chi_t^n]} \int_{\overline{ \mathcal{K}}} \left( M_{t,\mathcal{K}}^{-1} \chi_t + \ddbar (M_{t,\mathcal{K}}^{-1} V_{t,\mathcal{K}}) \right)^n\\
&\leq \frac{1}{[\chi_t^n]} \int_{\overline{ \mathcal{K}}} \left( M_{t,\mathcal{K}}^{-1} \chi_t + \ddbar (M_{t,\mathcal{K}}^{-1} V_{t,\mathcal{K}}) + (1- M_{t,\mathcal{K}}^{-1}) (\chi_t + \ddbar V_t) ) \right)^n  \\ 
&= \frac{1}{[\chi_t^n]} \int_{\overline{ \mathcal{K}}} (\chi_t + \ddbar\psi_{t, \mathcal{K}})^n  \\
&\leq \frac{ Cap_{\chi_t}(  \mathcal{K})}{[\chi_t^n]}. 
\end{split}\end{equation}

\medskip

\noindent{$\bullet$ \bf{The case when} $M_{t, K} \leq 1$.}  By \eqref{crt1}
$$0\leq  V_{t,\mathcal{K}} - V_t \leq \sup_X V_{t,\mathcal{K}}=M_{t,\mathcal{K}} \leq 1. $$
Now
\begin{equation}\label{eqn:3.18}[\chi_t^n] = \int_{\overline{ \mathcal{K}}} (\chi_t + \ddbar V_{t,\mathcal{K}})^n\leq Cap_{\chi_t}(\overline{ \mathcal{K}}). \end{equation} So in this case $\frac{[\chi_t^n]}{Cap_{\chi_t}( \mathcal{K})}\le 1$. 

Combining \eqref{eqn:3.17} and \eqref{eqn:3.18}, \eqref{eqn:MtK} holds and we complete the proof of Lemma \ref{lemma 3.1}.

\end{proof}

The following is an immediate corollary of  Lemma \ref{lemma 3.1}.
\begin{corr}\label{co 3.1}
There exists $C=C(X, \chi, \theta, p, K)>0$ such that for all $t\in (0,1]$, we have
$$\frac{1}{[\chi_t^n]}\int_ \mathcal{K} (\chi_t + \ddbar \varphi_t)^n \le C \bk{\frac{Cap_{\chi_t}( \mathcal{K})}{[\chi_t^n]}   }^2.$$
\end{corr}
\begin{proof}
This follows from Lemma \ref{lemma 3.1} and the elementary inequality that $x^2 e^{-\delta x^{1/n}}\le C$ for some uniform $C>0$ and  all $x\in (0,\infty)$.
\end{proof}

\begin{lemma}\label{lemma 3.2} Let $u \in \textnormal{PSH}(X, \chi_t) \cap L^\infty(X)$. For any $s>0$, $0\leq r \leq 1$ and $t\in (0,1]$, we have
\begin{equation}\label{eqn:lemma 3.4}
r^n Cap_{\chi_t} (u - V_t < -s-r) \leq \int_{\{u - V_t < -s\}} (\chi_t + \ddbar u)^n. 
\end{equation}

\end{lemma}

\begin{proof} For any $\phi \in \textnormal{PSH}(X, \chi_t)$ with $0\leq \phi - V_t \leq 1$, we have
\begin{eqnarray*}
r^n \int_{\{u - V_t < -s-r\}} (\chi_t + \ddbar \phi)^n &=& \int_{\{u-V_t< -s-r\}} (r\chi_t + \ddbar (r\phi))^n\\
&\leq& \int_{\{u-V_t < -s-r\}} (\chi_t + \ddbar (r\phi) + \ddbar (1-r) V_t)^n\\
&\leq& \int_{\{u-V_t< -s-r + r(\phi - V_t)\}}( \chi_t + \ddbar (r\phi + (1-r) V_t -s-r))^n\\
&\leq& \int_{\{u < r\phi+ (1-r)V_t -s-r\}} (\chi_t + \ddbar u)^n\\
&\leq& \int_{\{u < V_t-s\}} (\chi_t + \ddbar u)^n.
\end{eqnarray*}
The third inequality follows from the comparison principle and the last inequality follows from the fact that $r\phi + (1-r)V_t -s -r = r(\phi- V_t -1) + V_t -s <  V_t -s$.

 Taking supremum of all $\phi \in \textnormal{PSH}(X,\chi_t)$ with $0\le \phi - V_t\le 1$ we get \eqref{eqn:lemma 3.4}.

\end{proof}

\begin{lemma}\label{lemma 3.3}
Let $\varphi_t$ be the solution to \eqref{eqn:3.11}. Then there exists a constant $C = C(X,\chi, \theta, p, K)>0$ such that for all $s>1$
\begin{equation*}
\frac{1}{[\chi_t^n]}Cap_{\chi_t}\xk{ \{\varphi_t -V_t < -s\}  } \le \frac{C}{(s-1)^{1/q}},
\end{equation*}
where $\frac 1 p + \frac 1 q = 1$.

\end{lemma}
\begin{proof}
Applying Lemma \ref{lemma 3.2} to $u = \varphi_t$ and $r= 1$, we have 
\begin{align*}
& \frac{1}{[\chi_t^n]} Cap_{\chi_t}\xk{ \{ \varphi_t - V_t < -s \}   }\\
\le  &  \frac{1}{[\chi_t^n]} \int_{\{ \varphi_t - V_t < - (s-1)  \}} (\chi_t + \ddbar \varphi_t)^n\\
=&  \frac{1}{[\chi_t^n]} \int_{\{ \varphi_t - V_t < - (s-1)  \}} t^{n-\kappa}e^{-f + c_t} \theta^n\\
\le & \frac{C}{(s-1)^{1/q}}  \int_{\{ \varphi_t - V_t < - (s-1)  \}} ( - \varphi_t + V_t)^{1/q} e^{-f}\theta^n\\
\le & \frac{C}{(s-1)^{1/q}}\bk{  \int_{\{ \varphi_t - V_t < - (s-1)  \}}  e^{-p f }\theta^n}^{1/p} \bk{ \int_{ \{\varphi_t - V_t < -(s-1)\}  } (-\varphi_t + V_t)  \theta^n  }^{1/q}\\
\le & \frac{C}{(s-1)^{1/q}} \bk{ \int_X (-\varphi_t) \theta^n  }^{1/q},
\end{align*}
where in the last inequality we use the assumption that $e^{-f}\in L^p(\theta^n)$, $V_t\le 0$ and $\varphi_t\le 0$. On the other hand, since $\varphi_t\in \textnormal{PSH}(X, \chi_t)\subset \textnormal{PSH}(X, C \theta)$ for some large $C>0$ and $\sup_X\varphi_t = 0$, it follows from Green's formula that 
$$\int_X( -\varphi_t) \theta^n\le C $$ for some uniform constant $C$.
The lemma follows by combining the inequalities above.

\end{proof}

The following lemma is well-known and its proof can be found e.g. in \cite{K, EGZ1}. %But for completeness we provide the sketched proof below.
\begin{lemma}\label{lemma 3.4} Let $F:[0,\infty)\to [0,\infty)$ be a non-increasing right-continuous function satisfying
$\lim_{s\rightarrow \infty} F(s) =0$. 
If there exist $\alpha, A>0$ such that for all $s>0$ and $0\le r\le 1$, 
\begin{equation*}\label{eqn:assumption1} rF(s+r) \leq A \left( F(s)\right)^{1+\alpha},\end{equation*}
then there exists $S = S(s_0, \alpha, A)$ such that $$F(s) = 0$$ for all $s\ge S$, where $s_0$ is the smallest $s$ satisfying $\left( F(s)\right)^\alpha \leq (2A)^{-1}$.
\end{lemma}

%%%%%%%%%%%

\begin{proof}[Proof of Proposition \ref{main2}.]

Define for each fixed $t\in (0,1]$ $$F(s) = \bk{ \frac{Cap_{\chi_t} \xk { \{\varphi_t - V_t < -s   \}    }}{[\chi_t]^n}   }^{1/n}.$$
By Corollary \ref{co 3.1} and Lemma \ref{lemma 3.2} applied to the function $\varphi_t$, we have
$$r F(s+r)\le A F(s)^2,\quad \textnormal{ for ~ all} ~ r \in [0,1], \, s>0,$$
for some uniform constant $A>0$ independent of $t\in (0,1]$. 

Lemma \ref{lemma 3.3} implies that $\lim_{s\to \infty} F(s) = 0$ and the $s_0$ in Lemma \ref{lemma 3.4} can be taken as less than $(2 A C)^q$, which is a uniform constant. It follows from Lemma \ref{lemma 3.4} that $F(s) = 0$ for all $s>S$, where $S \le 2 + s_0$. On the other hand, if $Cap_{\chi_t}\xk{ \{\varphi_t - V_t < -s\}  } = 0$, by Lemma \ref{lemma 3.1} and the equation \eqref{eqn:3.11}, we have
$$\int_{\{ \varphi_t - V_t <-s  \}} e^{-f } \theta^n = 0,$$ 
hence the set $\{\varphi_t - V_t <-s\} = \emptyset$. Thus $\inf_X (\varphi_t - V_t) \ge - S$. Thus we finish the proof of Proposition \ref{main2}.

\end{proof}

Therefore we have proved Proposition \ref{main2} when $\lambda =0$.  We finish this section by proving the case when $\lambda=1$. We consider the following complex Monge-Amp\`ere equations for $t\in (0, 1]$, 
\begin{equation}
(\chi + t\theta + \ddbar \varphi_t)^n = t^{n-\kappa} e^{\varphi_t -f+ c_t} \theta^n. 
\end{equation}
where $f\in C^\infty(X)$ and $c_t$ is the normalizing constant satisfying $t^{n-\kappa} \int_X e^{-f+c_t} \theta^n = \int_X (\chi + t \theta)^n$.  

\begin{corr}\label{corr 3.2} 
If 
$$\|e^{-f}\|_{L^p(X, \theta^n)} \leq    K, $$
for $p>1$ and $ K>0$, 
Then there exists $C=C(X, \chi, \theta, p,  K)>0$ such that 

$$\| \varphi_t -V_t \|_{L^\infty} \leq C.$$

\end{corr}

\begin{proof} Since for each $t>0$, it is proved in \cite{BD} that $V_t $ is $C^{1,\alpha}(X, \theta)$, we can  always find  $W_t\in C^\infty(X)$ such that $\sup_X |V_t- W_t| \leq 1.$ Furthermore, $V_t$ is uniformly bounded above for all $t\in (0, 1]$. We let $\psi_t$ be the solution of 
$$(\chi_t + \ddbar \psi_t)^n = t^{n-\kappa} e^{-f+c_t +W_t } \theta^n, \quad \sup_X \psi_t =0.$$
and 
$$u_t = \varphi_t - \psi_t. $$ 
Then 
$$\frac{ (\chi_t + \ddbar \psi_t + \ddbar u_t)^n}{(\chi_t+ \ddbar \psi_t)^n}  = e^{u_t + \psi_t - W_t}. $$
Since $\sup_X |\psi_t - W_t|\leq \sup_X |\psi_t - V_t| + 1$, the maximum principle immediately implies that 
$$ \| u_t\|_{L^\infty(X)} \leq  \|\psi_t-V_t\|_{L^\infty(X)} + 1 $$
and so $$\|\varphi_t -V_t\|_{L^\infty(X)} \leq  2 \|\psi_t-V_t\|_{L^\infty(X)} + 1.$$

\end{proof}

%%%%%%%%%%%%%%%%%%%%%%%%%%%%%%%%

\section{Proof of Theorem \ref{main3}}

Let $X$ be a K\"ahler manifold. $X$ is said to be a minimal model if the canonical bundle $K_X$ is nef. The numerical dimension of $K_X$ is given by 
$$\nu(K_X) = \max \{ m = 0, ..., n~|~ [K_X]^m \neq 0 \text{ in }H^{m,m}(X,\mathbb C)\}. $$
 Let $\theta$ be a smooth K\"ahler form on a minimal model $X$ of complex dimension $n$. Let $\kappa= \nu(X)$, the numerical dimension of $K_X$. Let $\Omega $ be a smooth volume form on $X$. We let $\chi$ be defined by
 $$\chi = \ddbar \log \Omega\in K_X.$$
 We consider the following Monge-Amp\`ere equation for $t\in (0, \infty)$
 \begin{equation}\label{ke 1}
 (\chi + t\theta + \ddbar \varphi_t)^n = t^{n-\kappa} e^{\varphi_t} \Omega. 
 \end{equation}
 Since $K_X$ is nef, $[\chi + t\theta]$ is a K\"ahler class for any $t>0$. By Aubin and Yau's theorem, there exists a unique smooth solution $\varphi_t$ solving \eqref{ke 1} for all $t>0$. Let $\omega_t = \chi + t\theta + \ddbar \varphi$. Then $\omega_t $ satisfies
 $$ Ric(\omega_t) = - \omega_t + t \theta.$$
In particular, any K\"ahler metric satisfying the the above twisted K\"ahler-Einstein equation must coincide with $\omega_t$. 

 \begin{lemma} There exists $C>0$ such that for all $t\in (0, 1]$, 
$$C^{-1} t^{n-\kappa} \leq  [\chi+ t\theta]^n \leq C t^{n-\kappa}.$$ 
 
 \end{lemma} 
 
 \begin{proof} First we note that $[\chi]^d\cdot [\theta]^{n-\kappa} >0$ because $[\chi]^d\neq 0$ and $[\chi]$ is nef. Then
 $$[\chi+ t\theta]^n = t^{n-\kappa}\binom{n}{d} [\chi]^d\cdot [\theta]^{n-\kappa} + t^{n-\kappa+1} \Big( \sum_{j=d+1}^n \binom{n}{j} t^{j - d-1} [\chi]^j \cdot [\theta]^{n-j} \Big).$$

 \end{proof}

 \begin{lemma}\label{lemma 4.2} Let $V_t =  \sup\{ u ~|~ u\in \textnormal{PSH}(X, \chi + t\theta), ~  u \le 0\}  $. Then there exists $C>0$ such that for all $t\in (0, 1]$, 
 \begin{equation}
 \|\varphi_t - V_t\|_{L^\infty(X)} \leq C. 
 \end{equation}
 
 \end{lemma}
 
 \begin{proof} The lemma immediately follows by applying Proposition \ref{main2} to equation \eqref{ke 1}.

 \end{proof}

We now prove the main result in this section. 

\begin{lemma} \label{diam43} There exists $C>0$ such that for all $t\in (0, 1]$, 
$$Diam(X, g_t) \leq C. $$

\end{lemma}

\begin{proof}  The proof  applies similar argument in the proof of Theorem \ref{main1}. Suppose $Diam(X, g_t) = D$ for some $D\geq 6$. Let $\gamma:[0,D]\to X$ be a smoothing minimizing geodesic with respect to the metric $g_t$ and choose the points $\{x_i = \gamma(6i)\}_{i=0}^{[D/6]}$. It is clear that the balls $\{B_{g_t}(x_i,3)\}$ are disjoint so
$$\sum_{i=0}^{[D/6]} \vol_{\Omega}\xk{B_{g_t}(x_i,3)   }\le \int_X \Omega = V,$$
where $\vol_{\Omega}\xk{B_{g_t}(x_i, 3) }= \int_{B_{g_t}}(x_i, 3) \Omega. $
Hence there exists a geodesic ball $B_{g_t}(x_i, 3)$ such that 
$$\vol_{\Omega}\big(B_{g_t}(x_i, 3)\big) \leq 6V D^{-1}.$$ 
We fix such $x_i$ and  construct a cut-off function $\eta(x) = \rho(r(x))  \geq 0$ with $r(x) = d_{g_t}(x,x_i)$ such that 
$$\eta=1~ \textnormal{on}~ B_{g_t}(x_i, 1), \quad \eta=0~ \textnormal{outside}~ B_{g_t}(x_i, 2)$$
and
$$ \rho\in [0,1] ,\quad \rho^{-1} (\rho')^2\le C, \quad  |\rho''| \leq C. $$
Define a function $F_t>0$ on $X$ such that 
$$F_t=1~\textnormal{ outside}~B_{g_t} (x_i, 3), \quad F_t= D^{1/2} ~\textnormal{ on }~B_{g_t}(x_i, 2)$$
 and 
$$C^{-1} \leq \int_X F_t \Omega \leq C, \quad \int_X F_t^2 \Omega\leq C. $$

We now consider the equation
$$(\chi + t\theta+  \psi_t)^n = t^{n-\kappa}e^{ \psi_t} F_t\; \Omega,\quad \textnormal{for~ all~} t\in (0,1]. $$
Applying Corollary \ref{corr 3.2}, there exists a uniform constant $C>0$ such that for all $t\in (0,1]$, 
$$\|\psi_t- V_t\|_{L^\infty(X)} \leq C,$$
and so by Lemma \ref{lemma 4.2}
\begin{equation}\label{eqn:lemma 4.2}\|\varphi_t - \psi_t \|_{L^\infty(X)} \leq C. \end{equation}

Let $\hat g_t= \chi + t \theta_t + \ddbar \psi_t$. Then on $B_{g_t}(x_i, 2)$, 
$$Ric(\hat g_t) = - \hat g_t + t \theta, \quad Ric(g_t) = -  g_t + t\theta,  $$
and so 
$$ \Delta_{g_t} \log \frac{\hat\omega_t^n} {\omega_t^n} =- n + \tr_{g_t}(\hat g_t)\ge -n + n \bk{\frac{\hat \omega_t^n}{\omega_t^n}}^{1/n} . $$

Let $H = \eta \log \frac{\hat \omega_t^n} {\omega_t^n} $. We may suppose $\sup_X H =H(z_{max}) >0$, otherwise we are done. $z_{max}$ must lies in the support of $\eta$, and at $z_{max}$ we have 
\begin{align*}
0\ge \Delta_{g_t} H &\geq  \frac 1 \eta \bk{ H \Delta_{g_t} \eta + 2 \innpro{\nabla \eta,\nabla H} - 2 \frac{H}{\eta} \abs{\nabla \eta} - n\eta^2 + n\eta^2 e^{\frac{H}{n \eta}}     }\\
& \ge \frac 1 \eta \bk{ \frac 1{2n} H^2 - C H  }
\end{align*} 
for some uniform constant $C>0$ for all $t\in (0,1]$. Maximum principle implies that $\sup_X H \le C(n)$, in particular on $B_{g_t}(x_i,1)$ where $\eta \equiv 1$, there exists $C>0$ such that for all $t\in (0, 1]$, 
 $$\frac{\hat \omega_t^n}{\omega_t^n} = D^{1/2} e^{\psi_t - \varphi_t}\le C.$$ By the uniform $L^\infty$-estimate\eqref{eqn:lemma 4.2},  there exists  $C = C(n, \chi,\Omega,\theta)$ such that $D \le C$.

\end{proof}

Now we can complete the proof of Theorem \ref{main3}. Gromov's pre-compactness theorem and the diameter bound in Lemma \ref{diam43} immediately imply that after passing to a subsequence, $(X, g_{t_j})$ converges to a compact metric  space. Since $\varphi_t - V_t$ is uniformly bounded and $V_t $ is uniformly bounded below by $V_0$, $\varphi_{t_j}$ always converges weakly to some $\varphi_\infty \in \textnormal{PSH}(X, \chi)$, after passing to a subsequence. In particular, there exists $C>0$ such that  
$$ ||\varphi_\infty - V_0||_{L^\infty(X)} \leq C,$$
where $V_0$ is the extremal function on $X$ with respect to $\chi$. \qed

%%%%%%%%%%%%%%%%%%%%%%%%%%%%%%%%

\section{Proof of Theorem \ref{main4}}
Our proof is based on the arguments of \cite{ST1,To,TWY}. 

We fix some notations first. Recall $X_{can}$ has dimension $\kappa$ and $\chi$ is the restriction of Fubini-Study metric on $X_{can}$ from the embedding $X_{can}\hookrightarrow \mathbb{CP}^{N_m}$, where $N_m + 1 = \mathrm{dim} H^0(X,mK_X)$. Hence $\Phi^*\chi$ is a smooth nonnegative $(1,1)$-form on $X$, and in the following we identify $\chi$ with $\Phi^*\chi$ for simplicity. Let $\theta$ be a fixed K\"ahler metric on $X$.

Define a function $H\in C^\infty(X)$ as 
$$\chi^\kappa \wedge \theta^{n-\kappa} = H \theta^n$$
which is the modulus square of the Jacobian of the map $\Phi: (X,\theta)\to( X_{can},\chi)$ and vanishes on $S$, the indeterminacy set of $\Phi$, hence $H^{-\gamma}\in L^1(X,\theta^n)$ for some small $\gamma>0$. We fix a smooth nonnegative function $\sigma $ on $X_{can}$ as defined in \cite{To}, which satisfies 
\begin{equation}\label{eqn:sigma}
0\le \sigma\le 1,\quad 0\le \sqrt{-1}\partial \sigma\wedge \bar \partial \sigma\le C \chi, \quad -C\chi \le \ddbar \sigma \le C\chi,
\end{equation} 
for some dimensional constant $C=C(\kappa)>0$. From the construction, $\sigma$ vanishes exactly on $S' = \Phi(S)$. There exist $\lambda>0$, $C>1$ such that for any $y\in X_{can}^\circ = X_{can}\backslash S'$ (see \cite{To})
$$\sigma(y)^\lambda\le C \inf_{X_y} H,\quad \text{here }X_y = \Phi^{-1}(y). $$

The twisted K\"ahler-Einstein metric $g_t $ in \eqref{eqn:tKE} satisfies the following complex Monge-Amp\`ere equation (with $\theta = \theta$)
\begin{equation}\label{eqn:tKE1}
(\chi + t \theta + \ddbar \varphi_t)^n = t^{n-\kappa} e^{\varphi_t}\Omega,\quad \textnormal{for~ all~} t\in (0,1].
\end{equation} In case $K_X$ is semi-ample, $V_t = 0$ hence Corollary \ref{corr 3.2} implies: (see also \cite{DP, K, EGZ})

\begin{lemma}
There is a uniform constant $C>0$ such that 
$\| \varphi_t\|_{L^\infty(X)}\le C.$
\end{lemma}

We have the following Schwarz lemma whose proof is similar to that of Lemma \ref{lemma 2.2}, so we omit it.
\begin{lemma}\label{lemma:Sch}
There exists a constant $C>0$ such that
$$\tr_{\omega_t}\chi\le C,\quad \textnormal{for~ all~}~ t\in (0,1].$$
\end{lemma}

We denote $\theta_{y} = \theta|_{X_y}$ for $y\in X^\circ_{can}$, the restriction of $\theta$ on the fiber $X_y$ which is a smooth $(n-\kappa)$-dimensional Calabi-Yau submanifold of $X$. We will omit the subscript $t$ in $\varphi_t$ and simply write $\varphi = \varphi_t$, and define
$\overline{\varphi}_y = \dashint_{X_y} \varphi \theta^{n-\kappa}_{y}$ to be
the average of $\varphi$ over the fiber $X_y$. Denote the reference metric $\hat \omega_t = \chi + t\theta$. We calculate
$$(\hat \omega_t + \ddbar \varphi)|_{X_y} =\xk{ { t \theta_{y}}+ \ddbar (\varphi - \overline{\varphi}_y) }|_{X_y} = \omega_t|_{X_y},$$
hence
\begin{equation}\label{eqn:MA1}
\xk{ \theta_{y} + t^{-1} \ddbar (\varphi - \overline{\varphi}_y)|_{X_y}   }^{n-\kappa} = t^{-n+\kappa} \omega_{t,y}^{n-\kappa}.
\end{equation}
On the other hand,
\begin{align*}
t^{-n+\kappa}\frac{\omega_{t,y}^{n-\kappa}}{\theta_{y}^{n-\kappa}}& = t^{-n+\kappa}\frac{\omega_t^{n-\kappa}\wedge \chi^\kappa}{\theta^{n-\kappa}\wedge \chi^\kappa}\Big|_{X_y}\\
%&= t^{-n+\kappa}\frac{\omega_t^{n-\kappa}\wedge \chi^\kappa}{\omega_t^n} \frac{\omega_t^n}{\theta^{n-\kappa}\wedge \chi^\kappa}\Big|_{X_y}\\
&\le C \xk{\tr_{\omega_t}\chi}^\kappa \frac{\Omega}{\theta^{n-\kappa}\wedge \chi^\kappa}\Big|_{X_y}\\
&\le C H^{-1}\le C\sigma^{-\lambda}(y).
\end{align*}
Since the Sobolev constant of $(X_y, \theta_{y})$ is uniformly bounded and Poincar\'e constant of $(X_y,\theta_{y})$ is bounded by $C e^{B\sigma^{-\lambda}(y)}$ for some uniform constants $B,\, C>0$ (see \cite{To}), combined with the fact  that
$$\dashint_{X_y} (\varphi - \overline\varphi_y)\theta_{y}^{n-\kappa} = 0,$$
Moser iteration implies (\cite{Y1, To})
\begin{lemma}\label{lemma:1}There exist constants $B_1,\, C_1>0$ such that for any $y\in X^\circ_{can}$,
\begin{equation*}%\label{eqn:C0}
\sup_{X_y} t^{-1} |\varphi - \overline\varphi_y|\le C_1 e^{B_1\sigma^{-\lambda}(y)},\quad \textnormal{for~ all~} ~ t\in (0,1].
\end{equation*}
\end{lemma}

\begin{prop}\label{prop:2}
On any compact subset $K\Subset X\backslash S$, there exists a constant $C=C(K)>1$ such that for all $t\in (0,1]$
$$C^{-1}\hat \omega_t \le \omega_t \le C \hat \omega_t,\quad \text{on }K.$$
\end{prop}
Given  the $C^0$-estimate in Lemma \ref{lemma:1}, Proposition \ref{prop:2} can be proved by the $C^2$-estimate (\cite{Y1}) for Monge-Amp\`ere equation together with a modification as in \cite{ST1, To,TWY}, so we omit the proof.

%%%%%%%%%%%%%%%%%%%%%%%%%%%%%%%%%%%%%%%

%%%%%%%%%%%%%%%%%%%%%%%%%%%%%%%%%%%%

%\section{Proof of Proposition \ref{prop:1}}
Let us recall the construction of the canonical metric $\omega_{can}$ on $X_{can}^\circ$ (see \cite{ST1}). Define a function $F = \frac{\Phi_*\Omega}{\chi^\kappa}$ {on }$X_{can}^\circ,$
and $F$ is  in $L^{1+\varepsilon}$ for some small $\varepsilon>0$ (\cite{ST1}). The metric $\omega_{can}$ is obtained by solving the following complex Monge-Amp\`ere equation on $X_{can}$% (Bedford-Taylor sense) 
$$(\chi+\ddbar \varphi_\infty)^\kappa =\binom{n}{\kappa} F e^{\varphi_\infty} \chi^\kappa,$$
for $\varphi_\infty\in \textnormal{PSH}(X_{can},\chi)\cap C^0(X_{can})\cap C^\infty(X_{can}^\circ)$. Then $\omega_{can} = \chi + \ddbar \varphi_\infty$, and in the following we will write $\chi_\infty = \omega_{can}$. 

Any smooth fiber $X_y$ with $y\in X_{can}^\circ$ is a Calabi-Yau manifold hence there exists a unique Ricci flat metric $\omega_{SF,y}\in [\theta_{y}]$ such that $\omega_{SF,y} = \theta_{y}+ \ddbar \rho_y$ for some $\rho_y\in C^\infty(X_y)$ with normalization $\dashint_{X_y}\rho_y \omega^{n-\kappa}_{X,y} = 0$. We write $\rho_{SF}(x) = \rho_{\Phi(x)}$ if $\Phi(x)\in X_{can}^\circ$. $\rho_{SF}$ is a smooth function on $X\backslash S$ and may blow up near the singular set $S$. Denote $\omega_{SF} = \theta + \ddbar \rho_{SF}$ which is smooth on $X\backslash S$, and by \cite{ST1} we know that
$\frac{\Omega}{\omega_{SF}^{n-\kappa}\wedge \chi^\kappa}$
is constant on the smooth fibers $X_y$ and is equal to $\Phi^* F$. For simplicity we will identify $F$ with $\Phi^* F$. Our arguments below are motivated by \cite{ST1,TWY}.

Denote $\eE=e^{-e^{A\sigma^{-\lambda}}}$
for suitably large constants $A,\,\lambda>1$. From the proof of Proposition \ref{prop:2}, we actually have that on $X\backslash S$ (\cite{To})
\begin{equation*}C^{-1}\eE \hat\omega_t\le \omega_t\le C \eE^{-1}\hat \omega_t,\quad \textnormal{for~ all~} t\in (0,1].\end{equation*}
%\textcolor{blue}{Next we are going to show $\varphi_t\to \varphi_\infty = \Phi^* \varphi_\infty$ as $t\to\infty$ in some sense.} More precisely, we will prove:
Next we are going to show $\varphi_t\to \varphi_\infty = \Phi^* \varphi_\infty$ as $t\to 0$. Proposition \ref{prop:5.2} below can proved by following similar argument as in \cite{TWY}, but we present a slightly different argument in establishing {\bf Claim 2} below.
\begin{prop}\label{prop:5.2}
There exists a positive function $h(t)$ with $h(t)\to 0$ as $t\to 0$ such that 
\begin{equation}\label{eqn:uniform}\sup_{X\backslash S} \eE |\varphi_t - \varphi_\infty|\le h(t).\end{equation}
\end{prop}
\begin{proof}
Let $D\subset X_{can}$ be an ample divisor such that $X_{can}\backslash X_{can}^\circ \subset D$, $D\in \mu K_{X_{can}}$ for some $\mu\in\mathbb N$. Choose a continuous hermitian metric on $[D]$, $h_D = h_{FS}^{\mu/m} e^{-\mu\varphi_\infty}$ and a smooth defining section $s_D$ of $[D]$, where $h_{FS}$ is the Fubini-Study metric induced from $\mathcal O_{\mathbb {CP}^{N_m}}(1)$. Clearly $\ddbar \log h_D = \mu(\chi + \ddbar \varphi_\infty) =\mu\chi_\infty $. For small $r>0$, let $$B_r(D) = \{x\in X_{can}~|~d_\chi(x, D)\le r\}$$ be the tubular neighborhood of $D$ under the metric $d_\chi$, and denote $\mathcal B_r = \Phi^{-1}\xk{ B_r(D)  }\subset X$.

Since both $\varphi_t$ and $\varphi_\infty$ are bounded in $L^\infty$-norm, there exists $r_\epsilon$ with $\lim_{\epsilon\to 0}r_\epsilon = 0$ such that  for all $t\in (0,1]$
\begin{equation*}%\label{eqn:8.2}
\sup_{\mathcal B_{r_\epsilon}\backslash S}(\varphi_t - \varphi_\infty + \epsilon \log \abs{s_D}_{h_D})<-1,\quad\inf_{\mathcal B_{r_\epsilon}\backslash S} (\varphi_t - \varphi_\infty -\epsilon \log \abs{s_D}_{h_D})>1.
\end{equation*}
Let $\eta_\epsilon$ be a smooth cut-off function on $X_{can}$ such that $\eta_\epsilon = 1$ on $X_{can}\backslash B_{r_\epsilon}(D)$ and $\eta_\epsilon = 0$ on $B_{r_\epsilon/2}(D)$. Write $\rho_\epsilon = (\Phi^*\eta_\epsilon) \rho_{SF}$, and $\omega_{SF,\epsilon} = \omega_{SF} + \ddbar \rho_\epsilon$. Define the twisted differences of $\varphi_t$ and $\varphi_\infty$ by 
$$\psi_\epsilon^\pm = \varphi_t - \varphi_\infty - t \rho_\epsilon \mp \epsilon\log\abs{s_D}_{h_D}.$$ By similar argument in \cite{ST1} we have

%\begin{lemma}\label{lemma 8.1}T
{\bf Claim 1:} there exists an $\epsilon_0>0$ such that for any $\epsilon\in (0,\epsilon_0)$, there exists a $\tau_\epsilon$ such that for all $t\le \tau_\epsilon$, we have
$$\sup_{X\backslash S} \psi_\epsilon^-(t,\cdot)\le 3\mu\epsilon,\quad \inf_{X\backslash S} \psi_\epsilon^+(t,\cdot)\ge -3\mu\epsilon.$$
%\end{lemma}

%\begin{lemma}%\label{lemma 5.1}
{\bf Claim 2:} We have 
$$\int_X |\varphi_t - \varphi_\infty|\theta^n\to 0,\quad\text{as }t\to 0,$$
where $\varphi_t$ is the K\"ahler potential of $\omega_t$ in \eqref{eqn:tKE1}.
%\end{lemma} 
\begin{proof}[Proof of Claim 2]
For any $\eta>0$, we may take $\mathcal B_{R_\eta}\subset X$ small enough so that 
$\int_{\mathcal B_{R_\eta}} \theta^n < \frac \eta {10}.$ Take $\epsilon<\eta/10\mu$ small enough so that $r_\epsilon < R_\eta$. From {\bf Claim 1} when $t<\tau_\epsilon$
\begin{align*}
\int_X |\varphi_t - \varphi_\infty |\theta^n & =\int_{\mathcal B_{R_\eta}} |\varphi_t - \varphi_\infty |\theta^n+\int_{X\backslash \mathcal B_{R_\eta}} |\varphi_t - \varphi_\infty |\theta^n \\
&\le C\eta+ \int_{X\backslash \mathcal B_{R_\eta}}\xk{ t |\rho_{SF}| + \epsilon | \log\abs{s_D}_{h_D}  |  }\theta^n\\
&\le C\eta.
\end{align*}

\end{proof}
Given {\bf Claim 2}, Proposition \ref{prop:5.2} follows similarly as in \cite{TWY}, so we skip it.
\end{proof}
%\pagebreak

%%%%%%%%%%%%%%%%%%%%%%%%%%%%%%%%%%%%%%%%%%%%%%%%%

%%%%
We will apply an argument  in \cite{TWY} with a slight modification to show the lemma below:
\begin{lemma}
$$\lim_{t\to 0 }\eE  t \dot\varphi_t= 0.$$
\end{lemma} 
\begin{proof}Denote $s = \log t$ for $t\in (0,1]$. We have $t\dot \varphi = \frac{\partial \varphi}{\partial s}$. Taking derivatives on both sides of the equation \eqref{eqn:tKE1} and by maximum principle arguments we then get (see also \cite{TWY})
\begin{equation}\label{eqn:2nd upper}\frac{\partial^2 \varphi}{\partial s^2} = t \dot \varphi + t^2 \ddot \varphi \le C,\quad\text{here }\ddot \varphi = \frac{\partial^2\varphi}{\partial t^2}.\end{equation}
By the uniform convergence  \eqref{eqn:uniform} of $\eE \varphi(s)\to \eE \varphi_\infty$ as $s\to -\infty$, for any $\epsilon>0$, there is an $S_\epsilon$ such that for all $s_1,\, s_2\le -S_\epsilon$, we have  $ \sup_X |\eE \varphi(s_1) - \eE \varphi(s_2) |\le \epsilon$. For any $s< -S_\epsilon -1$ and $x\in X\backslash S$, by mean value theorem 
$$\eE \partial_s \varphi(s_x,x) = \frac{1}{\sqrt{\epsilon} } \int_s ^{s+\sqrt{\epsilon}} \partial_s ( \eE \varphi  )ds\ge -\sqrt{\epsilon},\quad \text{for some }s_x\in [s,s+\sqrt{\epsilon}].$$ By the upper bound \eqref{eqn:2nd upper}, it follows that
$\eE\partial_s \varphi(s,x)\ge- C \sqrt{\epsilon}  - \sqrt{\epsilon}.$
Similarly 
$$\eE \partial_s \varphi(\hat s_x,x) = \frac{1}{\sqrt{\epsilon}}\int_{s-\sqrt{\epsilon}}^s \partial_s(\eE \varphi(\cdot, x))ds\le \sqrt{\epsilon},\text{ for some }\hat s_x\in [s-\sqrt{\epsilon},s], $$ from \eqref{eqn:2nd upper} we get
$\eE \partial_s \varphi(s,x) \le C\sqrt{\epsilon} + \sqrt{\epsilon}.$ Hence we show that for any $s\le -S_\epsilon-1$ or $t=e^s\le e^{-S_\epsilon - 1}$, it holds that
$$\sup_{x\in X\backslash S} | \eE\partial_s \varphi(s,x) | = \sup_{x\in X\backslash S} |\eE t\partial_t \varphi(t,x)|\le C \sqrt{\epsilon},$$
so the lemma follows.

%%%%%%%%%%%%%%%%%%%%%%%%%%%%%%
\end{proof}
\begin{corr}\label{cor 5.1}
There exists a positive decreasing function $h(t)$ with $h(t)\to 0$ as $t\to 0$ such that
$$\sup_X \eE \xk{|\varphi_t - t \dot\varphi_t - \varphi_\infty| + t|\dot\varphi_t|}\le h(t).$$
\end{corr}
From Corollary \ref{cor 5.1} a straightforward adaption of the arguments of \cite{TWY}, we have an improvement of local $C^2$-estimate:
\begin{lemma}
On any compact subset $K\subset\subset X\backslash S$, we have
$$\limsup_{t\to 0}\bk{\sup_K \xk{ \tr_{\omega_t}\chi_\infty - \kappa  }}\le 0.$$
\end{lemma}

With the local $C^2$ estimate (see Proposition \ref{prop:2}), following standard local $C^3$-estimates (\cite{Y1, PSS,SW}), we have
\begin{lemma}\label{prop:Calabi}
For any compact $K\Subset X\backslash S$, there exists a $C= C(K)>0$ such that 
$$\sup_K  \abs{\nabla_{\theta} \omega_t  } \le C t^{-1}.$$ 
\end{lemma}

We have built up all the necessary ingredients to prove Theorem \ref{main4}, whose proof is almost identical to that of Theorem 1.3 in \cite{TWY}. For completeness, we sketch the proof below.
\begin{proof}[Proof of Theorem \ref{main4}]
Fix a compact subset $K'\subset X_{can}^\circ$ and let $K = \Phi^{-1}(K')$. By the Calabi $C^3$ estimate in Lemma \ref{prop:Calabi}, it follows that
\begin{equation*}%\label{eqn:final 1}
\| t^{-1} \omega_t |_{X_y}\|_{C^1(X_y,\theta_{y})}\le C,\quad t^{-1}\omega_t |_{X_y}\ge c\; \theta_{y},
\end{equation*}
for all $y\in K'$ and  $\theta_{y} = \theta|_{X_y}$. 

\begin{enumerate}[label={\bf Step \arabic*:}]
\item Define a function $f$ on $X_y$ by
$$f = \frac{(t^{-1}\omega_t|_{X_y})^{n-\kappa}}{\omega_{SF,y}^{n-\kappa}} = \binom{n}{\kappa}\frac{(\omega_t|_{X_y})^{n-\kappa}\wedge \chi_\infty^\kappa}{\omega_t^n}e^{\varphi_t -\varphi_\infty}\le e^{h(t)} \bk{\frac{\tr_{\omega_t}\chi_\infty}{\kappa}}^\kappa\le 1+ \tilde h(t), $$
for some $\tilde h(t)\to 0$ as $t\to 0$ (here $\tilde h(t)$ depends on $K$), where in the first inequality we use the Newton-Maclaurin inequality. $f$ also satisfies that
\begin{equation}\label{eqn:f} \int_{X_y} (f-1)\omega_{SF,y}^{n-\kappa} = 0,\quad \quad \lim_{t\to\infty}\int_{X_y} |f-1|\omega_{SF,y}^{n-\kappa} = 0.\end{equation}
The Calabi estimate implies that $\sup_{X_y}|\nabla f|_{\theta_{y}}\le C$ for all $y\in K'$, and $(X_y, \theta_{y})$ have uniformly bounded diameter and volume for $y\in K'$. So it follows that $f$ converges to $1$ uniformly on $K$ as $t\to 0$. That is
$$\| \big(t^{-1}\omega_t|_{X_y}\big)^{n-\kappa} - \omega_{SF,y}^{n-\kappa}\|_{C^0(X_y, \theta_{y})}\to 0,\text{ as }t\to 0,$$
uniformly on $K'$. Since $t^{-1}\omega_t|_{X_y}$ converges in $C^\alpha(X_y,\theta_{y})$ topology to some limit metric $\omega_{\infty,y}$ which satisfies the Monge-Amp\`ere equation (weakly) on $X_y$, $\omega_{\infty,y}^{n-\kappa} = \omega_{SF,y}^{n-\kappa}$, by the uniqueness of complex Monge-Amp\`ere equations, it follows that $\omega_{\infty,y} = \omega_{SF,y}$ and $t^{-1}\omega_t|_{X_y}$ converge in $C^\alpha$ to $\omega_{SF,y}$, for any $y\in K'$. Next we show the convergence is uniform in $K'$.

\medskip

\item Define a new $f$ on $X\backslash S$ which takes the form 
$$f|_{X_y} = \frac{t^{-1}\omega_t|_{X_y}\wedge (\omega_{SF,y})^{n-\kappa -1}}{\omega_{SF,y}^{n-\kappa}}\ge\bk{ \frac{(t^{-1}\omega_t|_{X_y})^{n-\kappa}}{\omega_{SF,y}^{n-\kappa}}}^{1/(n-\kappa)},$$
and the RHS tends to $1$ uniformly on $K$ as $t\to \infty$. Then we have similar equations as in \eqref{eqn:f} for this new $f$. This implies
$$\Big\| \frac{1}{n-\kappa}\tr_{\omega_{SF,y}}(t^{-1}\omega_t)|_{X_y} - 1 \Big\|_{L^\infty(K)}\to 0,\quad \text{as }t\to 0.$$
So $t^{-1}\omega_t|_{X_y}\to \omega_{SF,y}$ uniformly for any $y\in K'$.
 
\item Define $$\tilde \omega = t {\omega_{SF}} + \chi_\infty.$$From a result of \cite{TWY} (see the proof of Theorem 1.1 of \cite{TWY}), we have $|\tr_{\omega_t} \xk{ \omega_{SF} - \omega_{SF,y}  }|\le Ct^{-1/2}$, then
$$\tr_{\omega_t}\tilde \omega\le \tr_{\omega_t} \bk{ {t\omega_{SF,y}} + \chi_\infty   } + C \sqrt{t} = n + \tilde h(t),$$
for some $\tilde h(t)\to 0$ when $t\to 0$. Moreover it can be checked that
$$\lim_{t\to 0}\frac{\tilde \omega^n}{\omega_t^n}= 1, \quad\text{on }  K.$$
Hence we see that $\omega_t\xrightarrow{C^0(K)} \chi_\infty$ as $t\to 0$.

\end{enumerate}
We finish the proof of (1), (2) and (3) of Theorem \ref{main4}.
%\end{proof}

\begin{remark}
From {\bf Steps 1, 2 and 3}, we see that for any compact subset $K\subset X\backslash S$, there exists an $\varepsilon(t) = \varepsilon_K(t)\to 0$ as $t\to 0 $ such that when $t$ is small
\begin{equation}\label{eqn:rem 1}
\Phi^*\chi_\infty - \varepsilon(t)\theta\le \omega_t\le \Phi^*\chi_\infty + \varepsilon(t)\theta, \quad\text{on }K,
\end{equation}
and 
\begin{equation}\label{eqn:rem 2}
\Phi^*\chi_\infty \le (1+\varepsilon(t)) \omega_t,\quad \text{on }K.
\end{equation}
From the uniform convergence of $t^{-1}\omega_t|_{X_y}$ to $\omega_{SF,y}$ for any $y\in \Phi(K)$, we see that there is a uniform constant $C_0=C_0(K)>0$ such that
\begin{equation}\label{eqn:rem 3}
\omega_{t}|_{X_y}\le {C_0}t\omega_{SF,y},\quad \textnormal{for~ all~} y\in \Phi(K).
\end{equation}

\end{remark}

\newcommand{\zZ}{{\mathbf{Z}}}
%\section{Gromov-Hausdorff convergence}
Choose a sequence $t_k\to 0 $. The metric spaces $(X,\omega_{t_k})$ have $\ric(\omega_{t_k})\ge -1$ and $\mathrm{diam}(X,\omega_{t_k})\le D$ for some constant $D<\infty$. By Gromov's pre-compactness theorem up to a subsequence we have
$$(X,\omega_{t_k})\xrightarrow{d_{GH}} (\zZ, d_\zZ),$$ 
for some compact metric length space $\zZ$ with diameter bounded by $D$. The idea of the proof of (4) in Theorem \ref{main4} is motivated by \cite{GTZ}, and we present below a slightly different argument from theirs.
%\begin{lemma}
\smallskip

\noindent{\bf  Step 4:} We will show {\bf Claim 3:} There exists an open subset $\zZ_0\subset \zZ$ and a homeomorphism $f: X_{can}^\circ \to \zZ_0$ which is a local isometry.
%\end{lemma}
\begin{proof}[Proof of {\bf Claim 3}]
By Lemma \ref{lemma:Sch}, the maps $\Phi = \Phi_k: (X,\omega_{t_k})\to (X_{can},\chi)$ are uniformly Lipschitz with respect to the given metrics, and the target space is compact, so up to a subsequence $\Phi_k\to \Phi_\infty: (\zZ,d_\zZ)\to (X_{can}, \chi)$ along the GH convergence $(X, \omega_{t_k})\to (\zZ,d_\zZ)$ which is also Lipschitz and the convergence is in the sense that for any $x_k\to (X,\omega_{t_k})$ which converges to $z\in \zZ$, then $\Phi_\infty(z) = \lim_{k\to\infty} \Phi_k(x_k)$, and there is a constant $C>0$ such that  $d_\chi\xk{ \Phi_\infty(z_1), \Phi_\infty(z_2)  }\le C d_\zZ (z_1,z_2)$ for all $z_i\in \zZ$.

\smallskip

We denote $\zZ_0 = \Phi_\infty^{-1}(X_{can}^\circ)$ which is an open subset of $\zZ$ since $\Phi_\infty$ is continuous. We will show that $\Phi_\infty|_{\zZ_0}: \zZ_0\to X_{can}^\circ$ is a bijection and a local isometry. Hence $f =( \Phi_\infty|_{\zZ_0})^{-1}: X_{can}^\circ \to \zZ_0$ is the desired map.  

\smallskip

$\bullet$ {\em $\Phi_\infty|_{\zZ_0}$ is injective:} Suppose $\Phi_\infty(z_1) = \Phi_\infty(z_2)$ for $z_1,z_2\in \zZ_0 = \Phi_\infty^{-1}(X_{can}^\circ)$. Denote $y = \Phi_\infty(z_1) = \Phi_\infty(z_2)\in X_{can}^\circ$. Since $(X_{can}^\circ,\chi_\infty)$ is an (incomplete) smooth Riemannian manifold there exists a small $r = r_y>0$ such that $\big(B_{\chi_\infty}(y, 2r), \chi_\infty \big )$ is geodesic convex. Choose sequences $z_{1,k}$ and $z_{2,k}\in (X,\omega_{t_k})$ converging $z_1, z_2$ respectively along the GH convergence. By definition of $\Phi_k = \Phi \to\Phi_\infty$ it follows that $d_\chi (\Phi(z_{1,k}), \Phi_\infty(z_1)  )\to 0 $ and $d_\chi ( \Phi(z_{2,k}), \Phi_\infty(z_2) )\to 0$. Since $d_{\chi}$ and $d_{\chi_\infty}$ are equivalent on $B_{\chi_\infty}(y, 2r)$, it follows that $d_{\chi_\infty}(\Phi(z_{1,k}), \Phi(z_{2,k})   ) \to 0$ and hence we can find minimal $\chi_\infty$-geodesics $\gamma_k$ connecting $\Phi(z_{1,k})$ and $\Phi(z_{2,k})$ with $\gamma_k\subset B_{\chi_\infty}(y, r)$  and $L_{\chi_\infty}(\gamma_k) \to 0$. By the locally uniform convergence \eqref{eqn:rem 1} on $\Phi^{-1}\xk{ \overline{ B_{\chi_\infty}(y, 2r)  }  }$
there exists a lift of $\gamma_k$, $\tilde \gamma_k$ in $\Phi^{-1}\xk{ \overline{ B_{\chi_\infty}(y, 2r)  }  }$, such that 
$L_{\omega_{t_k}} (\tilde \gamma_k)\le L_{\chi_\infty}(\gamma_k) + \epsilon(t_k) L_{\omega}(\tilde \gamma_k) \to 0$ as $t_k\to 0$. $\tilde \gamma_k$ connects $z_{1,k}$ and $z_{2,k}$ hence $d_{\omega_{t_k}}(z_{1,k}, z_{2,k}) \le L_{\omega_{t_k}} (\tilde \gamma_k)\to 0$, which implies by the convergence of $z_{i,k}\to z_i$ that $d_\zZ(z_1, z_2) = 0$ and $z_1 = z_2$.

\smallskip

$\bullet$ {\em $\Phi_\infty|_{\zZ_0}$ is a local isometry:} let $z\in \zZ_0$ and $y = \Phi_\infty(z)\in X_{can}^\circ$. There is a small $r=r_y>0$ such that $(B_{\chi_\infty}(y, 3r), \chi_\infty)$ is geodesic convex. Take $U = (\Phi_\infty|_{\zZ_0})^{-1} \xk{ B_{\chi_\infty}(y, r)  }$
 to be an open neighborhood of $z\in \zZ$. We will show that $\Phi_{\infty}|_{\zZ_0}: (U, d_\zZ)\to (B_{\chi_\infty} (y, r), \chi_\infty)$ is an  isometry. Fix any two points $z_1, z_2\in U$ and $y_i = \Phi_\infty(z_i)\in B_{\chi_\infty}(y, r)$ for $i = 1, 2$. As before we choose $z_{i,k}\in (X,\omega_{t_k})$ such that $z_{i,k}\to z_i$ along the GH convergence for $i = 1,2$. It follows then from $\Phi_k = \Phi\to \Phi_\infty$ that $d_{\chi_\infty}\xk{ \Phi(z_{i,k}), y_i  }\to 0$, and when $k$ is large, $\Phi(z_{i,k}) $ lie in $B_{\chi_\infty}(y, 1.1r)$. Choose $\omega_{t_k}$-minimal geodesics $\gamma_k$ connecting $z_{1,k}$ and $z_{2,k}$ such that $d_{\omega_{t_k}}(z_{1,k},z_{2,k})  = L_{\omega_{t_k}}(\gamma_k) \to d_\zZ(z_1,z_2)$. The curve $\bar \gamma_k = \Phi(\gamma_k)$ connects $\Phi(z_{1,k})$ with $\Phi(z_{2,k})$. If $\bar \gamma_k\subset B_{\chi_\infty}(y, 3r)$, from \eqref{eqn:rem 2} it follows that 
 $$d_{\chi_\infty}(\Phi(z_{1,k}), \Phi(z_{2,k}))\le L_{\chi_\infty}(\bar \gamma_k) \le (1+\epsilon(t_k)) L_{\omega_{t_k}}(\gamma_k)\to d_\zZ(z_1, z_2).$$
In case $\bar \gamma_k \not\subset B_{\chi_\infty}(y, 3r)$, we have  
$$d_{\chi_\infty}(\Phi(z_{1,k}), \Phi(z_{2,k}))\le 3.8 r \le L_{\chi_\infty}(\bar \gamma_k \cap B_{\chi_\infty}(y, 3r)  )\le (1+\epsilon(t_k)) L_{\omega_{t_k}}(\gamma_k)\to d_\zZ(z_1, z_2).$$ Letting $k\to \infty$ we conclude that $d_{\chi_\infty}(y_1,y_2)\le d_\zZ(z_1, z_2)$. To see the reverse inequality, we take $\chi_\infty$-minimal geodesics $\sigma_k$ connecting $\Phi(z_{1,k})$ and $\Phi(z_{2,k})$. Clearly $\gamma_k\subset B_{\chi_\infty}(y, 3r)$. Take a lift of $\sigma_k$, $\tilde \sigma_k$ in $\Phi^{-1}\xk{ \overline{ B_{\chi_\infty}(y, 3r) }  }$ it follows from \eqref{eqn:rem 1} that $d_{\omega_{t_k}} (z_{1,k}, z_{2,k})\le  L_{\omega_{t_k}}(\tilde \sigma_k) \le L_{\chi_\infty}(\sigma_k) + \epsilon(t_k) L_\omega(\tilde \sigma_k)\to d_{\chi_\infty}(y_1, y_2)$. Letting $k\to\infty$ we get $d_\zZ(z_1, z_2) \le d_{\chi_\infty}(y_1, y_2)$. Hence $d_\zZ(z_1, z_2) = d_{\chi_\infty}(y_1, y_2)$ and $\Phi_\infty|_{\zZ_0} : U \to B_{\chi_\infty}(y, r)$ is an isometry.

\smallskip

$\bullet$ {\em$\Phi_\infty|_{\zZ_0}$ is surjective:} this is almost obvious from the definition. Take any $y\in X_{can}^\circ$ and any fixed point $x\in \Phi^{-1}(y)\subset (X,\omega_{t_k})$. Up to a subsequence $x\xrightarrow{d_{GH}} z\in (\zZ,d_\zZ)$. It then follows from $\Phi_k\to \Phi_\infty$ that $d_\chi(y, \Phi_\infty(z)) = d_{\chi}\xk{ \Phi_k(x), \Phi_\infty(z)  } \to 0$ as $k\to\infty$. Hence $\Phi_\infty(z) = y$ and $z\in \Phi_\infty^{-1}(X_{can}^\circ) = \zZ_0$.

\end{proof}

%%%%%%%%%%%%%%%%%%%%%%%%%%%

\smallskip

\noindent{\bf Step 5:}
In this step we will show $\zZ_0\subset \zZ$ is dense. Fix a base point $\bar x\in \zZ_0$, upon rescaling if necessary we may assume the metric ball $B_{\chi_\infty}\xk{ f^{-1}(\bar x),2  }\subset (X_{can}^\circ,\chi_\infty)$ is geodesic convex. Choose a sequence of points $\bar p_k\in (X,\omega_{t_k})$ such that $\bar p_k\to \bar x$ along the GH convergence $(X,\omega_{t_k})\to (\zZ, d_\zZ)$. We define a function on $X\times [0,\infty)$ as the normalized volume (\cite{CC})
$$\underline{V}_k (x,r) = \frac{ \vol_{\omega_{t_k}}\xk{ B_{\omega_{t_k}}(x,r) } }{ \vol_{\omega_{t_k}}\xk{B_{\omega_{t_k}} (\bar p_k, 1) } },$$
by standard volume comparison it is shown in \cite{CC} that $\underline{V}_k(\cdot,\cdot)$ is equi-continuous and uniformly bounded hence they converges (up to  a subsequence) to a function $\underline{V}_\infty: \zZ\times [0,\infty)\to [0,\infty)$ in the sense that for any $x_k\to x$ along the GH convergence and $r\ge 0$,
$$\underline{V}_k(x_k,r)\to \underline{V}_\infty(x,r),\quad \text{as }k\to\infty.$$
And $\underline{V}_\infty$ satisfies similar estimates as in volume comparison, i.e. for $r_1\le r_2$, $\frac{\underline{V}_\infty(x, r_1)}{\underline{V}_\infty(x,r_2)}\ge \mu(r_1,r_2)>0$ where $\mu(\cdot, \cdot)$ is the quotient of volumes of balls in a space form. 
The function $\underline{V}_\infty$ induces a Radon $\nu$ on $(\zZ,d_\zZ)$. More precisely for any $K\subset \zZ$, define
$$\hat\nu(K) = \lim_{\delta\to 0}\hat\nu_\delta(K) = \lim_{\delta\to 0}\inf  \sum_i \underline{V}_\infty(x_i,r_i) $$
where the infimum is taken over all metric balls $B_{d_\zZ}(x_i,r_i)$ with $r_i\le \delta$ whose union covers $K$.

\smallskip

\noindent{\bf Claim 4:} %\begin{lemma}
 For any $x\in \zZ_0$ and $r=r_x >0$ such that $B_{\chi_\infty}\xk{f^{-1}(x),2r  }\subset X_{can}^\circ$ is geodesic convex, we have
$$\underline{V}_\infty(x,r) = v_0 \int_{ \Phi^{-1}\xk{ B_{\chi_\infty}(f^{-1}(x), r)  }  } e^{-\varphi_\infty}\theta^n$$
for a fixed constant $v_0 = \bk{\int_{\Phi^{-1}\xk{B_{\chi_\infty}(f^{-1}(\bar x),1)   }  } e^{\varphi_\infty}\theta^n    }^{-1}.  $
%\end{lemma}
\begin{proof}[Proof of {\bf Claim 4}]
The proof is parallel to that in \cite{GTZ}, so we only provide a sketch. For the given $x\in \zZ_0$, we choose a sequence of points $p_k\in (X,\omega_{t_k})$ such that $p_k\to x$. 
As in \cite{GTZ}, due to \eqref{eqn:rem 1} and that the metrics $\omega_{t_k}$ and $\theta$ are equivalent in $\Phi^{-1}\xk{ B_{\chi_\infty}(f^{-1}(x), 2r)  }$, it can be shown that
\begin{equation}\label{eqn:inclusion}\Phi^{-1}\xk{B_{\chi_\infty}(f^{-1}(x), r - \epsilon_k)   } \subset B_{\omega_{t_k}} \xk{ p_k, r  }\subset \Phi^{-1}\xk{B_{\chi_\infty}(f^{-1}(x), r + \epsilon_k)   }\end{equation}
when $k>>1$ and here $\epsilon_k\to 0$ as $k\to\infty$.  It follows then that
$$\lim_{k\to\infty}\int_{B_{\omega_{t_k}}(p_k, r)  } e^{\varphi_{t_k}}\theta^n = \int_{\Phi^{ -1  }\xk{ B_{\chi_\infty}(f^{-1}(x), r)  }    } e^{\varphi_\infty}\theta^n.   $$
From the equation $\omega_t^n = t^{n-\kappa} e^{\varphi_t} \theta^n$, we have
\begin{align*}
\underline{V}_k(p_k,r) = & \frac { \int_{B_{\omega_{t_k}}(p_k,r)} t^{n-\kappa} e^{\varphi_{t_k}}\theta^n     }{ \int_{B_{\omega_{t_k}}(\bar p_k, 1)   } t_k ^{n-\kappa} e^{\varphi_{t_k}}\theta^n  } \\
\to & \frac{ \int_{ \Phi^{-1}\xk{ B_{\chi_\infty}(f^{-1}(x), r)  }} e^{\varphi_\infty} \theta^n    }{ \int_{\Phi^{-1}\xk{ B_{\chi_\infty}(f^{-1}(\bar x), 1)   }} e^{\varphi_\infty} \theta^n   },
\end{align*}
where for the convergence of the denominators we use a similar relation as in  \eqref{eqn:inclusion} for $\bar p_k,\, \bar x$. From the definition that $\underline{V}_k(p_k, r)\to \underline {V}_\infty(x,r)$, we finish the proof of {\bf Claim 4}.
\end{proof}
Since along the Gromov-Hausdorff convergence the diameters are uniformly bounded by $D<\infty$, $\vol_{\omega_{t_k}}(B_{\omega_{t_k}}(p_k, D)) = \vol(X,\omega_{t_k}^n)$. So 
$$\underline{V}_\infty(x, D) = \lim_{k\to\infty} \frac{ \vol_{\omega_{t_k}} ( B_{\omega_{t_k}}(p_k, D)  )   }{\vol_{\omega_{t_k}} ( B_{\omega_{t_k}} (\bar p_k, 1)   )   } = \lim_{k\to\infty} \frac{ \int_X e^{\varphi_{t_k}}\theta^n   }{ \int_{B_{\omega_{t_k}} (\bar p_k, 1)  } e^{\varphi_{t_k}} \theta^n } = v_0 \int_X e^{\varphi_\infty}\theta^n.  $$ Therefore from $\zZ = B_{d_\zZ}(x, D)$, we have $$\hat\nu(\zZ)\le v_0 \int_X e^{\varphi_\infty} \theta^n.$$ 

\smallskip

Assume $\zZ_0\subset \zZ$ were not dense, then there exists a metric ball $B_{d_\zZ}(z, \rho)\subset \zZ\backslash \zZ_0$, by volume comparison estimate for $\underline{V}_\infty$
$$\hat\nu\xk{ B_{d_\zZ}(z, \rho)  }\ge \underline{ V }_\infty(z, D) \mu(\rho, D)=: \eta_0>0.$$
Then for any compact subset $K\subset \zZ_0$, $\hat\nu(K)\le \hat\nu(\zZ) - \eta_0$. On the other hand, for any open covering $B_{d_\zZ} (x_i, r_i)$ of $K$ with $B_{\chi_\infty}(f^{-1}(x_i), 2r_i)$ geodesic convex in $(X_{can}^\circ, \chi_\infty)$ and $r_i<\delta$, we have
$$\sum_i \underline{V}_\infty(x_i,r_i) = \sum_i v_0 \int_{ \Phi^{-1}\xk{ B_{\chi_\infty}(f^{-1}(x_i), r_i)  }  } e^{\varphi_\infty} \theta^n\ge v_0 \int_{\Phi^{-1} \xk{ f^{-1}(K)  }  } e^{\varphi_\infty}\theta^n$$ 
taking infimum over all such coverings and letting $\delta\to 0$, we get
$\hat\nu(K)\ge v_0 \int_{\Phi^{-1} \xk{ f^{-1}(K)  }  } e^{\varphi_\infty}\theta^n$. If we take $K$ large enough so that $f^{-1}(K)\subset X_{can}^\circ$ is large, we can achieve that 
$$\hat\nu(K)\ge v_0 \int_{\Phi^{-1}(X_{can}^\circ)} e^{\varphi_\infty}\theta^n - \frac{\eta_0}{10} = v_0 \int_{X} e^{\varphi_\infty}\theta^n - \frac{\eta_0}{10}\ge \hat\nu(\zZ) - \frac{\eta_0}{10}.$$
Hence we get a  contradiction, and $\zZ_0\subset \zZ$ is dense since $\hat\nu(\zZ\backslash \zZ_0) = 0$. %\qed

\end{proof}

%%%%%%%%%%%%%%%%%%%%%%%%%%%%%%%%%%%%%%%

\section{Proof of Theorem \ref{main5}}

The proof of Theorem \ref{main5} is almost identical with that of Theorem \ref{main4}. We give the sketch here. The solution $g_t$ lies in the K\"ahler class $tL + (1-t)K_X$ for all $t\in ( t_{min}, 1]$. By definition and straightforward calculations from estimates of Yau \cite{Y1} and Aubin \cite{A},  for any $t\in (t_{min}, 1]$, the class $t L + (1-t) K_X$ is K\"ahler and so  $ t_{min} L + (1-t_{min})K_X$ is nef. We let $\Omega$ be a smooth volume form on $X$ and $\chi \in [  t_{min} L +  (1-t_{min})K_X]$ be a smooth closed $(1,1)$-form defined by
$$ \chi = \ddbar \log \Omega +  \theta. $$
Then the twisted K\"ahler-Einstein equation (\ref{curveq}) is equivalent to the following complex Monge-Amp\`ere equation for $t\in (t_{min}, 1]$
\begin{equation}\label{finitetime} 
(\chi + ( t - t_{min} ) \theta  + \ddbar \varphi_t  )^n = (t- t_{min} )^{n - \kappa} e^{\varphi_t} \Omega, 
\end{equation}
where $\kappa = \nu( t_{min} L +  (1-t_{min})K_X)$, the numerical dimension of the line bundle $t_{min} L + (1 - t_{min}) K_X$.  By Proposition \ref{main2}, there exists $C=C(X, \chi, \theta)>0$ such that for all $t\in (t_{min}, 1]$, 
$$\|\varphi_t -V_t \|_{L^\infty(X)} \leq C, $$
where $V_t$ is the extremal function associated to $\chi + (t -t_{min} ) \theta$. 
The rest of the proof for Theorem \ref{main5} is exactly the same as that of Theorem \ref{main3} and we leave it as an exercise for interested readers. 

%%%%%%%%%%%%%%%%%%%%%%

\bigskip

\noindent {\bf{Acknowledgements:}} 
This is part of the thesis of Xin Fu at Rutgers University, and he would like to thank the Department of Mathematics for its generous support. Bin Guo would like to thank Prof. D.H. Phong for many stimulating discussions and his constant support and encouragement. The authors thank Valentino Tosatti for helpful comments on the earlier draft.

%%%%%%%%%%%%%%%%%%%%%%


\begin{thebibliography}{100}

\bibitem{A} Aubin, T.  {\em \'Equations du type Monge-Amp\`ere sur les vari\'et\'es k\"ahl\'eriennes compactes},  Bull. Sci. Math. (2) {\bf 102} (1978), no. 1, 63--95

\bibitem{BD}Berman, R. and Demailly, J.-P. {\em Regularity of plurisubharmonic upper envelopes in big cohomology classes}. Perspectives in analysis, geometry, and topology, 39--66, Progr. Math., 296, Birkh\"auser/Springer, New York, 2012


\bibitem{B} Blocki, Z. {\em A gradient estimate in the Calabi-Yau theorem}, Math. Ann. 344 (2009), 317--327

\bibitem{BEGZ} Boucksom, S., Eyssidieux, P., Guedj, V. and Zeriahi, A. {\em Monge-Amp\`re equations in big cohomology classes}. Acta Math. 205 (2010), no. 2, 199--262



\bibitem{CC}Cheeger, J. and Colding, T. {\em On the structure of spaces with Ricci curvature bounded below. I}. J. Differential Geom. 46 (1997), no. 3, 406--480

\bibitem{Da} Datar, V. {\em On convexity of the regular set of conical K\"ahler-Einstein metrics}, Math. Res. Lett. 23 (2016), no. 1, 105--126

\bibitem{DDGHKZ}Demailly, J-P,  Dinew, S., Guedj, V., Pham, H., Kolodziej, S., Zeriahi, A.
{\em H\"older continuous solutions to Monge-Amp\`ere equations}. J. Eur. Math. Soc. (JEMS) 16 (2014), no. 4, 619 -- 647


\bibitem{DP}Demailly, J-P. and Pali, N. {\em Degenerate complex Monge-Amp\`ere equations over compact
K\"ahler manifolds}, Internat. J. Math. 21 (2010), no. 3, 357--405.


\bibitem{EGZ}  Eyssidieux, P., Guedj, V. and Zeriahi, A. {\em A priori $L^\infty$-estimates for degenerate
complex Monge-Amp\`ere equations}, Int. Math. Res. Not. IMRN 2008, Art. ID rnn
070, 8 pp.

\bibitem{EGZ1}Eyssidieux, P., Guedj, V. and Zeriahi, A.  {\em Singular K\"ahler-Einstein metrics}. J. Amer. Math. Soc. 22 (2009), no. 3, 607--639

\bibitem{GTZ}Gross, M., Tosatti, V. and Zhang, Y. {\em Collapsing of abelian fibered Calabi-Yau manifolds}. Duke Math. J. 162 (2013), no. 3, 517--551

\bibitem{GTZ1}Gross, M., Tosatti, V. and Zhang, Y. {\em Gromov-Hausdorff collapsing of Calabi-Yau manifolds}. Comm. Anal. Geom. 24 (2016), no. 1, 93--113

\bibitem{GS} Guo, B. and Song, J. {\em An analytic base point free theorem}, in preparation

\bibitem{Ka} Kawamata, Y. {\em Pluricanonical systems on minimal algebraic varieties}, Invent. Math. 79 (1985), no. 3, 567--588



\bibitem{K} Kolodziej, S. {\em The complex Monge-Amp\`ere equation}. Acta Math., 180 (1998), 69--117

\bibitem{K2} Kolodziej, S. {\em H\"older continuity of solutions to the complex Monge-Amp\`ere equation with the right-hand side in $L^p$: the case of compact K\"ahler manifolds},  Math. Ann. 342 (2008), no. 2, 379--386

\bibitem{KT} Kolodziej, S. and Tian, G. {\em   A uniform $L^\infty$-estimate for complex Monge-Amp\`ere equations},  Math. Ann. 342 (2008), no. 4, 773--787

\bibitem{LT} La Nave, G. and Tian, G. {\em  A continuity method to construct canonical metrics},  Math. Ann. 365 (2016), no. 3-4, 911--921

\bibitem{LTZ} La Nave, G., Tian, G. and Zhang, Z. {\em  Bounding diameter of singular K\"ahler metric},  arXiv:1503.03159



\bibitem{PSS} Phong, D. H.,  Sesum, N. and Sturm, J.  {\em Multiplier ideal sheaves and the K\"ahler-Ricci
flow}, Comm. Anal. Geom. 15 (2007), no. 3, 613--632

\bibitem{PS}Phong, D. H, and Sturm, J. {\em On pointwise gradient estimates for the complex Monge-Amp\`ere equation}. Advances in geometric analysis, 87--95, Adv. Lect. Math. (ALM), 21, Int. Press, Somerville, MA, 2012

\bibitem{SW} Sherman, M. and Weinkove, B. {\em Interior derivative estimates for the K\"ahler-Ricci flow},
Pacific J. Math. 257 (2012), no. 2, 491--501

\bibitem{S0}Song, J. {\em Ricci flow and birational surgery}, arXiv:1304.2607

\bibitem{S1} Song, J. {\em Riemannian geometry of K\"ahler-Einstein currents}, 
arXiv:1404.0445

\bibitem{S2} Song, J. {\em Riemannian geometry of K\"ahler-Einstein currents II: an analytic proof of KawamataÕs base point free theorem}, arXiv:1409.8374

 \bibitem{S3} Song, J. {\em Degeneration of K\"ahler-Einstein manifolds of negative scalar curvature}, preprint


 \bibitem{SW1} Song, J. and Weinkove, B. {\em Contracting exceptional divisors by the K\"ahler-Ricci flow}, Duke Math. J. 162 (2013), no. 2, 367--415
   
 \bibitem{SW2} Song, J. and Weinkove, B. {\em Contracting exceptional divisors by the K\"ahler-Ricci flow II}, Proc. Lond. Math. Soc. (3) 108 (2014), no. 6, 1529--1561
   
   

\bibitem{ST1}  Song, J. and Tian, G. {\em The K\"ahler-Ricci flow on surfaces of positive Kodaira dimension},
Invent. Math. 170 (2007), no. 3, 609--653

\bibitem{ST2} Song, J. and Tian, G. {\em Canonical measures and K\"ahler-Ricci flow}, J. Amer. Math. Soc.
25 (2012), no. 2, 303--353

\bibitem{ST3} Song, J. and Tian, G. {\em The K\"ahler-Ricci flow through singularities}, Invent. Math. 207 (2017), no. 2, 519--595


\bibitem{SY}  Song, J. and Yuan, Y. {\em Metric flips with Calabi ansatz}, Geom. Func. Anal. 22 (2012), no. 1, 240--265
 
 
\bibitem{T1}Tian, G. {\em On K\"ahler-Einstein metrics on certain K\"ahler manifolds with $C_1(M)>0$}. Invent. Math. 89 (1987), no. 2, 225--246


\bibitem{TW} Tian, G. and Wang, B. {\em On the structure of almost Einstein manifolds},  J. Amer. Math. Soc. 28 (2015), no. 4, 1169--1209


\bibitem{To}Tosatti, V. {\em Adiabatic limits of Ricci-flat K\"ahler metrics}, J. Differential Geom. 84
(2010), no.2, 427--453

\bibitem{TWY}Tosatti, V., Weinkove, B. and Yang, X. {\em The K\"ahler-Ricci flow, Ricci-flat metrics and collapsing limits}, preprint arXiv:1408.0161, to appear in Amer. J. Math. 2017



\bibitem{Y1} Yau, S.-T. {\em On the Ricci curvature of a compact K\"ahler manifold and the complex
Monge-Amp\`ere equation}, I, Comm. Pure Appl. Math. 31 (1978), 339--411


\bibitem{Y2} Yau, S.T. {\em A general Schwarz lemma for K\"{a}hler manifolds},
 Amer. J. Math.  100 (1978), 197--204
 
\bibitem{ZZh} Zhang, Y. and Zhang, Z. {\em The continuity method on minimal elliptic K\"ahler surfaces}, 	arXiv:1610.07806
 
\bibitem{ZZ} Zhang, Z. {\em On degenerate Monge-Amp\`ere equations over closed K\"ahler manifolds},  Int. Math. Res. Not. 2006, Art. ID 63640, 18 pp.



\end{thebibliography}
\end{document}